\documentclass[12pt,twoside]{amsart}
\usepackage{latexsym,amsmath,amsopn,amssymb,amsthm,amsfonts}
\usepackage[T2A]{fontenc}
\usepackage[cp1251]{inputenc}
\usepackage[english]{babel}
\usepackage{graphicx}

\newtheorem{theorem}{Theorem}

\newtheorem{proposition}{Proposition}

\newtheorem{remark}{Remark}

\newtheorem{definition}{Definition}

\newtheorem{example}{Example}

\newtheorem{corollary}{Corollary}

\newtheorem{question}{Question}

\DeclareMathOperator{\Ad}{Ad}

\DeclareMathOperator{\ad}{ad}

\DeclareMathOperator{\SO}{SO}

\DeclareMathOperator{\SU}{SU}

\DeclareMathOperator{\GL}{GL}

\DeclareMathOperator{\sgn}{sign}

\DeclareMathOperator{\SL}{SL}

\DeclareMathOperator{\M}{M}

\DeclareMathOperator{\const}{const}

\DeclareMathOperator{\Inn}{Inn}

\DeclareMathOperator{\pr}{pr}

\begin{document}
	
\vspace{0.5cm}
\title[Sub-Lorentzian geodesics on $\GL^+(2,\mathbb{C})$]{Sub-Lorentzian geodesics on $\GL^{+}(2,\mathbb{C})$ with the generating space of Hermitian matrices in the Lie algebra $\mathfrak{gl}^{+}(2,\mathbb{C})$}
\author{V.~N.~Berestovskii, I.~A.~Zubareva}
\thanks{The work of the first author was carried out within the framework of the State Contract to the IM SB RAS, project FWNF-2022-0006.
The work of the second author was carried out within the framework of the State Contract to the IM SB RAS, project FWNF--2022--0003.}
\address{Sobolev Institute of Mathematics of the SB RAS,\newline
4 Koptyug Av., Novosibirsk, 630090, Russia}
\email{valeraberestovskii@gmail.com}
\address{Sobolev Institute of Mathematics of the SB RAS,\newline
13 Pevtsova st., Omsk, 644099, Russia}
\email{i\_gribanova@mail.ru}

\begin{abstract}
The Lie subgroup $\GL^+(2,\mathbb{C})$ of all matrices in the Lie group $\GL(2,\mathbb{C})$ 
with positive real determinant is equipped with a left-invariant
sub-Lorentzian (anti)metric defined by the natural structure of the 4-dimensional Minkowski space-time on the subspace of Hermitian matrices in its Lie algebra. 
On base of the corresponding time--anti--optimal control problem, formulated in the paper, and Pontryagin minimum principle for it, using
geodesics and shortest arcs of the corresponding left-invariant sub-Riemannian metric on the Lie subgroup $\SL(2,\mathbb{C})$, the authors found sub-Lorentzian
nonspacelike geodesics and longest arcs.
\vspace{2mm}

\vspace{1mm}

\end{abstract}
\maketitle

\begin{flushright}
Dedicated to the 90-th anniversary of Anatoly Moiseevich Vershik
	
\end{flushright}

\section{Preface}

The first author highly appreciates all communications with Anatoly Moiseevich.

They include the first correspondence acquaintance through his publications on sub-Riemannian manifolds, especially the joint
survey \cite{VerGer87} with V.Ya.Gershkovich, my talk at Leningrad on my thesis \cite{Ber90}, my defence of this dissertation in 
1990 at Novosibirsk, when Anatoly Moiseevich was its opponent, the work for our joint paper \cite{BerVer92}, and meetings in May 2004 at 
Max-Planck-Institute f\"ur Mathematik, Bonn.  

Really all the text \cite{Ber90} was printed by its author with old-fashioned typewriter yet in 1987. It includes a general theory of locally
compact homogeneous spaces $M$ with intrinsic metric $\rho$. Any space $(M,\rho)$ is {\it geodesic} in Gromov sense, i.e. every its pair
of points can be joined by a shortest path (curve). This follows from Cohn-Vossen theorem which states that any two points of a locally compact
complete intrinsic metric space can be joined by a shortest path \cite{CV36}. If $(M,\rho)$ is infinite-dimensional, then it is an 
inverse Gromov-Hausdorff limit of a sequence of homogeneous manifolds with invariant intrinsic metrics (connected by submetries \cite{BerGui00}),
while the latter ones are homogeneous spaces with invariant (sub-)Finslerian metric. The space $(M,\rho)$ admits a metric similarity with a coefficient
$\alpha\neq 1$ onto itself if and only if $(M,\rho)$ is isometric either to a finite-dimensional normed vector space or to a simply connected graded nilponent Lie group, the so-called {\it Carnot group} with a special left-invariant sub-Finslerian metric \cite{Ber90, BerVer92}. E.~Le Donne published later a proof of this result in \cite{Led15}.

Let us indicate as examples of the Carnot groups the following ones: 3-dimensional Heisenberg group, 4-dimensional Engel group, and 5-dimensional Cartan group.

On the other hand, there is no real possibility for a general theory of non-locally-compact separable or non-separable homogeneous spaces with intrinsic metrics because of immense amount of such spaces. Between them are the famous geodesic and complete universal separable Urysohn's space $U$ \cite{Ur27}, \cite{Ur51} and complete geodesic non-separable homogeneous $\mathbb{R}$-tree $RT$ with the valency continuum $\mathfrak{c}$ at any point. The space $RT$ has infinite Hausdorff dimension, the covering (Lebesgue) dimension 1, and the curvature $\leq K$ for every $K\in\mathbb{R}$ in the sense of A.D.~Aleksandrov. Both spaces
$U,$ $RT$ admit metric similarities onto themselves with any coefficient $\alpha\in\mathbb{R}_+.$ The space $RT$ is an universal space for the $\mathbb{R}$-trees
of valency $\leq\mathfrak{c}.$ 

Surprisingly, P.S.~Urysohn constructed in \cite{Ur27a}, \cite{Ur51a} possibly the very first example of a rather nontrivial $\mathbb{R}$-tree $\mathcal{R}$.
The space $\mathcal{R}$ has {\it all properties of $RT$} ($\mathcal{R}$ can be interpreted as a group with left-invariant metric) {\it besides the completeness}. If $\overline{\mathcal{R}}$ is the completion of $\mathcal{R},$ then $\overline{\mathcal{R}}\setminus\mathcal{R}$ is everywhere dence subset in $\overline{\mathcal{R}}$ and $\overline{\mathcal{R}}$ is a $\mathbb{R}$-tree. In fact, Berestovskii in 1989, as well as I.V.~Polterovich and A.I.~Shnirelman in 1997, constructed the $\mathbb{R}$-trees which are isometric to $\mathcal{R},$ unaware then with Urysohn's example from \cite{Ur27a}, \cite{Ur51a}, although the note \cite{Ur51a} is placed immediately after \cite{Ur51}. Even more surprisingly that paper \cite{DP98}, containing the construction of $RT,$ (as well as paper \cite{Ver98} by A.M.~Vershik on the space $U$ and Gromov metric triples) is published in the issue of Russian Math. Surveys dedicated to 100 years of Urysohn birthday, and in this issue there is no mention on the important sense of notes \cite{Ur27a}, \cite{Ur51a}. All the above and other information on $\mathbb{R}$-trees with many references is given in \cite{Ber19}.

The Urysohn space $U$ was used in \cite{BerVer92} for a definition of the Gromov-Hausdorff distance between separable metric spaces. Paper \cite{BerVer92} was cited
in \cite{Ver98}, \cite{Gro96}, \cite{Semmes01}. In very interesting paper \cite{Ver04}, A.M.~Vershik introduced the notion of random metric space and proved that such a space is isometric to the Urysohn universal metric space $U$ with probability one. He also quoted there many others interesting and important connected results obtained  by him and other authors.   

\section{Introduction} 

The Lie group $\GL^+(2,\mathbb{C})$ is a subgroup of the Lie group $\GL(2,\mathbb{C})$ of all non-degenerate complex $(2\times 2)$--matrices, consisting of the matrices with positive real determinant. The Lie group  $\GL^+(2,\mathbb{C})$ is $7$--dimensional and includes the $6$--di\-men\-si\-on\-al Lie group $\SL(2,\mathbb{C})$ of all unimodular complex $(2\times 2)$--matrices.

In Theorem 2 from \cite{Ber23}, the well-known two-sheeted universal covering--epimorphism $l:\SL(2,\mathbb{C})\rightarrow \SO_0(1,3)$ with the kernel 
$\{I,-I\}\in \SU(2)$ onto the special orthochronic Lorentz group $\SO_0(1,3)$ is constructed. 
$\SO_0(1,3)$ is the unit connected component of the complete group of linear Lorentz transformations of the
Minkowski space-time $M_0.$ This is consistent with the fact that the noncompact symmetric Riemannian space $\SL(2,\mathbb{C})/\SU(2)=\SO_0(1,3)/\SO(3)$
of type IV \cite{Hel} is a 3-dimensional Lobachevsky space $L^3,$ since one of orbits of the group $\SO_0(1,3)$ is the upper sheet of the hyperboloid of two sheets
in $M_0$ with pseudoscalar product $\langle\cdot,\cdot\rangle$ of the  signature $(-,+,+,+),$ which is a model of the space $L^3.$

In the preface to book \cite{RumFet77}, it is noted that replacing the Lorentz group with a two-sheeted covering $\SL(2,\mathbb{C})$ leads to simplifications in  spinor algebra. The extension of $l$ to $\GL^+(2,\mathbb{C})$ is an epimorphism with the same kernel of the group $\GL^+(2,\mathbb{C})$ to the transitive unit connected component of the  group of all conformal ($\cong$ causal) transformations of the open future cone $C_0\subset M_0$ \cite{Ber23}, \cite{PanSeg82}.

It is interesting that the 4-dimensional linear subspace $H$ of all Hermitian matrices in the Lie algebra $\mathfrak{gl}^+(2,\mathbb{C})$ of the Lie group
$\GL^+(2,\mathbb{C})$ generates $\mathfrak{gl}^+(2,\mathbb{C})$. The subspace $H$ equipped with the quadratic form $4\det h,$ $h\in H,$ of the signature
$(+,-,-,-)$, is isometric to the space-time $M_0.$ This, together with a time orientation, defines a left-invariant sub-Lorentzian structure on the Lie group
$\GL^+(2,\mathbb{C})$ from Remark 5 in \cite{Ber23}. In this paper, we study geodesics, i.e. locally longest arcs, of this sub-Lorentzian structure and properties of such geodesics.

{\it A left-invariant (sub-)Lorentzian structure} on a connected Lie group $G$ is given by some pseudoscalar product $\langle\cdot,\cdot\rangle$ with the signature $(+,-,\dots,-)$ on the Lie algebra $( \mathfrak{g},[\cdot,\cdot])$ of the Lie group $G$, inducing
 a pseudoscalar product with similar signature on the subspace $\mathfrak{p}\subset \mathfrak{g}$ generating $\mathfrak{g}$.
The definition of such structure is completed by choosing a timelike vector $v\in\mathfrak{p}$ with the condition $\langle v, v\rangle =1,$ defining {\it time orientation} on $G.$
 
Every left-invariant (sub-)Lorentzian structure on a Lie group $G$ induces a {\it left-invariant (sub-)Lorentzian (anti)metric} $d$ on $G$.

Here $d(x,y)$ is the supremum of lengths of piecewise continuously differentiable {\it  future directed timelike horizontal paths} $g=g(t),$ $0\leq t \leq a,$ joining $x,y$, i.e., the paths with conditions
\begin{equation}
\label{cond} 
g(0)=x,\, g(a)=y,\, g'(t)=dl_{g(t)}(u(t)),\, u(t)\in\mathfrak{p},\, \langle u(t),v\rangle>0,\, \langle u(t),u(t)\rangle >0,
\end{equation}
where $dl_{g(t)}$ is the differential of the left shift $l_{g(t)}: h\in G\rightarrow g(t)h;$ the length of each such path is determined by the standard formula $L(g)=\int_0^a\sqrt{\langle u(t),u(t)\rangle}\ dt.$

Unlike the metric, $d$ satisfies the inequality $$d(x,z)\geq d(x,y)+d(y,z),$$ which is opposite to the triangle inequality;
in general case, the equalities $d(x,y)=-\infty$ and $d(x,y)=+\infty$ are possible.
Those cases where the latter equality is impossible, are of particular interest; these cases
exclude the existence of so-called future directed timelike loops.

The (sub-)Lorentzian (anti)metric on the Lie group $G$ with the condition $d<+\infty$ 
is a very special case of the so-called {\it intrinsic (anti)metric} on locally compact topological groups (which may also have infinite topological dimension), defined by the axiomatic AM, equivalent to the other two axiomatics MO (of the {\it metrized order}) and SG (of the {\it subgraph-semigroup}), \cite{BerGich99}. Many examples of metrized orders are presented in \cite{BerGich01}.
Here the fundamental role is played by the concept of a semigroup, including a semigroup-family of
subsets on a topological group.

The emerging search problem for timelike future directed {\it longest arcs} of a left-invariant intrinsic (anti)metric on a connected Lie group $G$
in some sense is dual to the search problem for shortest paths of a left-invariant intrinsic metric on $G$. The last mentioned problem is a {\it left-invariant time-optimal problem with a control region}, which is some compact convex centrally symmetric body $W\subset \mathfrak{p},$ where $\mathfrak{p}$ is a vector subspace of the Lie algebra $ (\mathfrak{g},[\cdot,\cdot]),$ generating $\mathfrak{g}$ by the operation $[\cdot,\cdot]$.

In Theorems 11 and 12 from \cite{BerGich01}, the necessary conditions were formulated for the search for the so-called normal shortest arcs (timelike
longest arcs) of a left-invariant intrinsic metric ((anti)metric) on a Lie group $G.$
Theorem 11 is proved in Theorem 7 in \cite{BerZub}. Let us formulate Theorem 12.

\begin{theorem}
\label{bergich}

1. Each parametrized by arclength timelike longest arc $g(t),$ $0\leq t\leq a,$ of left-invariant 
inner (anti)metric $\tau$ on a Lie group $G$ with the Lie algebra $\mathfrak{g}$, defined by an antinorm $\nu$
on the subspace $\mathfrak{p}\subset \mathfrak{g}$ with closed unit ball $U_1$, is a Lipschitzian
time-anti-optimal trajectory of the control system $dl_{g(t)^{-1}}(g'(t))=u(t)\in U_1.$

2. (Left-invariant Pontryagin minimum principle). 
Every such longest arc is a Lipschitzian integral curve of (=tangent to for almost all $t\in [0,a]$) the
differential inclusion
\begin{equation}
\label{incl}	
\{d_el_g(u)\in T_gG\mid \psi_g(u)=\min\{\psi_g(v)\mid v\in U_1\}\},
\end{equation}
where $\psi_g:=(\Ad g)^{\ast}(\psi_0)=\psi_0\circ \Ad g\in \mathfrak{g}^{\ast}$, and $\psi_0\in\mathfrak{g}^{\ast}$ is a fixed covector such that  $\min\{\psi_0(v)\mid v\in U_1\}=1.$
If the set $U_1$ is strictly convex then the differential inclusion (\ref{incl}) is a continuous vector field on $G.$

3. (Conservation law). Furthermore, $\psi(t)(dl_{g(t)^{-1}}(g'(t)))=1$ for almost all $t\in [0,a]$ and $\psi(t)=\psi_{g(t)}.$
\end{theorem}	 

Let us explain that the antinorm $\nu$ has properties, similar to those of the antinorm $\|u\|:=\sqrt{\langle u,u\rangle},$ $U_1$ 
is one of the closed convex connected components of the set
$$V_1=\{v\in\mathfrak{p}\mid \nu(v)\geq 1\}.$$

In the case of a Lie group $G$ with a left-invariant pseudoscalar product $\langle\cdot,\cdot\rangle$ the antinorm
$\nu(u):=\|u\|=\sqrt{\langle u,u\rangle}$ is considered. If $\langle u,u\rangle>0$ then
$$\|\alpha u\|=|\alpha|\|u\|,\, \alpha\in\mathbb{R};\, \|u_1+u_2\|\geq \|u_1\|+\|u_2\|,\,\mbox{if}\ \langle u_1,u_2\rangle>0.$$
Therefore, the region $U_1$ is convex, moreover, strictly convex, closed, but unlike $W,$ it is not compact.
Note also that cones
$$\mathcal{C}=\{u\in \mathfrak{p}\mid\langle u,u\rangle>0, \langle u, v\rangle>0\}$$
and $\overline{\mathcal{C}}$, containing $U_1$, are open and closed
semigroups with respect to addition.

\begin{remark}
The items 2 and 3 of Theorem \ref{bergich} are related to Theorem 5 of Chapter 12 in \cite{Jur97}. 	
\end{remark} 

In general case of a left-invariant sub-Lorentzian structure on a connected Lie group $G$,  parametrized by arclength timelike longest arcs may be 
{\it abnormal}, when $\psi(u(t))=\min\{\psi(t)(u)\mid u\in U_1\}\equiv M_0=0.$

A special form of the Pontryagin maximum principle (together with the corres\-pon\-ding Hamiltonian system) for left-invariant (sub-)Riemannian metrics on Lie groups is derived in \cite{BerZub}, \cite{Ber14} and is based on the Pontryagin maximum principle for the time-optimal problem from \cite{Pontr}.

Similarly, on the base of ideas used in  Theorem \ref{bergich}, in Theorem \ref{main} is formulated in detail somewhat differently
the {\it Pontryagin minimum principle} for timelike curves and the {\it conjugate ordinary differential equations} for a nonzero covector function $\psi=\psi(t),$ $t\in \mathbb{R},$ giving the mentioned necessary conditions for solutions of the specified problem, as well as their analogues in Theorem \ref{isotrn} 
for isotropic curves with a different control region.

The curves satisfying these conditions are called {\it extremals}.

In this paper, we found  {\it nonspacelike}, i.e. {\it timelike or isotropic}, extremals, and prove that 
so-called {\it normal}
(and {\it nonstrictly abnormal}) extremals of the above left-invariant sub-Lorentzian structure on the Lie group $\GL^+(2,\mathbb{C})$ are geodesic.

There are quite a lot of papers on left-invariant sub-Lorentzian structures on Lie groups.
Apparently, the first such paper is by M.~Grochovsky \cite{Gr02},
except for the earlier paper \cite{BerGich01}, where in examples 16 and 17 was considered a sub-Lorentzian structure on the Heisenberg group (in different terms) and  descriptions of timelike longest arcs, different from segments of 1-parameter subgroups and their shifts, and a subgraph of the antinorm for structure were given.

We have found no clear and simple formulation of analogues to the Pontryagin principle in papers on sub-Lorentzian geometry except in the recent paper \cite{Sach23} by Yu.L.~Sachkov and E.F.~Sachkova, which includes many references to these papers.

In Sec. \ref{subriem} and \ref{subriem1} we study geodesics of left-invariant sub-Riemannian metric $\rho$ on the Lie group
$\SL(2,\mathbb{C})$ defined by the scalar product $(\cdot,\cdot)$ on the space $H_0=H\cap \mathfrak{sl}(2,\mathbb{C}),$ where $(\cdot,\cdot)$ is a restriction of the pseudoscalar product $-\langle\cdot,\cdot\rangle$ to $H_0.$ It follows from paper \cite{Ber18} that all such geodesics are normal and are products of at most two 1-parameter subgroups if they start at the unit.

In Theorem \ref{ppp} we prove that every 1-parameter subgroup with  initial tangent vector from $H_0$ is a metric line. 
Due to the proof of Theorem \ref{ppp}, the mapping $\exp:H_0\rightarrow \exp(H_0)$ is a diffeomorphism.  The set $\exp(H_0)$ consists exactly of
all positive definite Hermitian matrices in $\SL(2,\mathbb{C}).$ The set $\exp(H_0)\setminus \{I\}$ 
coincides with the set of all {\it boosts} relative to the standard time axis in Minkowski space-time $M_0$ (see Sec. \ref{append}).
They generate the group $\SL(2,\mathbb{C}).$ Proposition \ref{p1} presents sub-Riemannian geodesics 
for which is established an upper bound for the lengths of their shortest segments.
In Proposition \ref{osn}, we establish intersections of sub-Riemannian geodesics with $\exp(H_0)$.
The proof of Proposition \ref{osn} is not easy, but Proposition is justified by its connection with boosts and positive definite Hermitian matrices.

In Sec. \ref{apple}, using sub-Riemannian geodesics and shortest arcs in $(\SL(2,\mathbb{C}),\rho)$,  
we characterize future oriented nonspacelike longest arcs (Theorem \ref{slgeod}) and geodesics in  $(\GL^+(2,\mathbb{C}),d)$ with origin at the unit.
In Theorem \ref{slgeod} is proved that if $e$ has no connection to $g\in \GL^+(2,\mathbb{C})$ by a longest arc, then there is no future directed nonspacelike horizontal curve which joins $e$ and $g.$

In Theorems \ref{main2} and \ref{main3} we prove that each normal nonspacelike sub-Loretzian extremal on
$\GL^+(2,\mathbb{C})$ with origin at the unit is realized as a product of at most two 1-parameter subgroups.

An exact matrix form for normal nonspacelike extremals is established in Corollary
\ref{form} on the base of Theorems \ref{main2}, \ref{main3} and Proposition \ref{expon}. In Proposition, \ref{isom} we prove that the action of the subgroup $\SU(2)\subset \GL^+(2,\mathbb{C})$ on the space $(\GL^+(2,\mathbb{C}),d)$ by conjugations is isometric. This allows, up to isometries, to simplify the matrix form for extremals.

In particular, all normal sub-Lorentzian extremals are geodesic (Theorem \ref{slgeod1}), and every segment of a future directed nonspacelike 1-parameter subgroup with an initial vector from $H$ is a longest arc (Proposition \ref{slpr} and Theorem \ref{onepa}).
In Sec. \ref{abn}, we proved that all nonstrictly abnormal extremals on $\GL^+(2,\mathbb{C})$ with origin at the unit are 1-parameter subgroups, as in
Proposition \ref{slpr} and Theorem \ref{onepa}. Therefore, each of their segments is a longest arc in  $(\GL^+(2,\mathbb{C}),d).$ 

In the last section of the paper, we present important connections of the problem under study with some concepts and results from mathematics and Lorentzian geometry. At the end of this paper one unsolved question is raised.

\section{Preliminaries}
\label{res}

Let $G$ be a connected Lie group with the Lie algebra  $(\mathfrak{g},[\cdot,\cdot])$,  $\mathfrak{p}\subset\mathfrak{g}$ be a subspace, generating the Lie algebra $\mathfrak{g}$ by the operation $[\cdot,\cdot]$, $\langle\cdot,\cdot\rangle$ be a pseudoscalar product with the signature   $(+,-,\dots, -)$ on $\mathfrak{g}$ 
and an orthonormal basis $e_0,\dots, e_n$ in the Lie algebra  $\mathfrak{g}$ such that  $e_0,\dots, e_r$ is an orthonormal basis of the subspace $\mathfrak{p}$, $v:=e_0,$ $\langle v, v\rangle= 1.$ 
The corresponding left-invariant pseudoscalar product on $TG$ we also denote by $\langle\cdot,\cdot\rangle$.

Denote by $\Delta$ the left-invariant distribution on $G$ such that $\Delta(e)=\mathfrak{p}$. A vector 
$w\in T_{g}G$, $g\in G$, is called horizontal if $w\in\Delta(g)$, i.e., $(dl_{g})^{-1}(w)\in\mathfrak{p}$, where $l_{g}:h\in G\rightarrow g\cdot h$ is the
left shift on the group $(G,\cdot)$ by an element $g$. A Lipschitzian curve $g(t)$, $t\in [0,a]$, in $G$ is called horizontal if $g'(t)\in\Delta(g(t))$ for almost all $t\in [0,a]$. Since $\mathfrak{p}\subset\mathfrak{g}$ generates the Lie algebra $\mathfrak{g}$, then, by the Chow--Rashevskii theorem, any two points $g_1,g_2\in G$ can be joined by a horizontal curve. 

Further we shall consider only horizontal vectors and curves. A vector $w$ is called:

1) timelike if $\langle w,w\rangle >0$;

2) spacelike if $\langle w,w\rangle <0$ or $w=0$;

3) isotropic if $\langle w,w\rangle=0$ and $w\neq 0$;

4) nonspacelike if $\langle w,w \rangle\geq 0$.

A horizontal curve $g(t)$, $t\in [0,a]$, in $G$ is called timelike if $\langle g'(t),g'(t)\rangle>0$ for almost all $t\in [0,a]$; 
spacelike, isotropic and nonspacelike horizontal curves are defined in a similar way.

A nonspacelike vector $w\in\Delta(g)$ is {\it future directed  (resp. past directed)} if $\langle (dl_{g})^{-1}(w),v\rangle>0$ 
(resp.  $\langle (dl_{g})^{-1}(w),v\rangle<0$).

The length of a nonspacelike curve $g(t)$, $t\in [0,a]$, is given by the formula
$$L(g)= \int\limits_{0}^{a}\sqrt{\langle g'(t),g'(t)\rangle}dt.$$

For any points $g_0,g_1\in G$ denote by $\Omega_{g_0g_1}$ the set of all future directed nonspacelike curves  
$g(t)$, $t\in [0,a]$ (i.е., $g'(t)$ is future directed for almost all $t\in [0,a]$), in $G$ that join $g_0=g(0)$ to $g_1=g(a)$. If $\Omega_{g_0g_1}\neq\varnothing$, then the sub-Lorentzian distance from  $g_0$ to $g_1$ is equal to
\begin{equation}
\label{co}
d(g_0,g_1)=\sup\{L(g)\mid g\in\Omega_{g_0g_1}\}.
\end{equation}
If  $\Omega_{g_0g_1}=\varnothing$ then we put $d(g_0,g_1)=-\infty$.

A future directed nonspacelike curve $g(t)$, $t\in [0,a]$, of the sub-Lorentz space $(G,d)$ is called a longest arc if  it realizes the supremum in (\ref{co}) between its endpoints  $g(0)=g_0$ and $g(a)=g_1$. A Lipschitzian curve
$g(t)$, $t\in\mathbb{R}$, in $(G,d)$ is called a (sub)Lorentz geodesic if locally it is a longest arc.

A timelike longest arc $g(t)$, $0\leq t\leq a=d(g_0,g_1)$, in $(G,d)$ with $g(0)=g_0$, $g(a)=g_1$, is a time--anti--optimal solution to the control system
\begin{equation}
\label{system}
g'(t)=dl_{g(t)}(u(t)),\quad u(t)\in U=\left\{u\in\mathfrak{p}\mid\langle u,u\rangle\geq 1,\,\,\langle u, v\rangle>0\right\},
\end{equation}
with indicated endpoints.

{\it The time--anti--optimal control problem} is to find a measurable control $u(t)$, $t\in [0,a]$, such that the corresponding Lipschitzian trajectory $g(t )$, $t\in [0,a]$, joins points $g_0$ and $g_1$, and the transition time $a$ from $g_0$ to $g_1$ is maximum.

At first, one needs to prove the existence of nonspacelike longest arcs of (sub-)Lo\-ren\-t\-zian space $(G,d)$ joining certain pairs of points from $G.$

After that, to find them, {\it Pontryagin minimum principle} is used for the time--anti--optimal control problem and a covector function
$\psi=\psi(t)\in T^{\ast}_{g(t)}G$, which gives only the necessary conditions in general case. 
The covector function can be considered as a left-invariant 1-form on $(G,\cdot)$ and identified with function $\psi(t)\in\mathfrak{g}^{\ast}=T_e^{\ast}G.$

Every optimal timelike trajectory $g(t),$ $0\leq t\leq T,$ is determined by some measu\-rable optimal control $u=u(t)\in U,$
$0\leq t\leq T,$ and for an absolutely continuous non-vanishing function $\psi=\psi(t),$ $0\leq t\leq T,$ we have for almost all $t\in [0, T],$ 
\begin{equation}
\label{sg}
\dot{g}(t)=dl_{g(t)}(u(t)),
\end{equation} 
\begin{equation}
\label{hame1}
\psi(w)^{\prime}= \psi([u(t),w]),\,w\in \mathfrak{g},
\end{equation}
\begin{equation}
\label{max}
M(t):=\psi(t)(u(t))=\min_{u\in U}\psi(t)(u):=M_0\geq 0
\end{equation}

\begin{definition}
A timelike extremal for the problem (\ref{system}) is a parameterized by arc\-length future directed curve 
$g = g(t)$, $t\in \mathbb{R},$ satisfying the Pontryagin minimum principle for the time--anti--optimal problem.	
An extremal is called normal (abnormal) if $M_0 > 0$ ($M_0=0$).
Every normal extremal $g = g(t)$, $t\in \mathbb{R},$ is parameterized by the arclength, i.e., $\langle g'(t),g'(t)\rangle=1$ for almost all $t\in\mathbb{R}$; proportionally changing $\psi=\psi(t),$ $t\in \mathbb{R},$ if it is necessary, one can assume that $M_0=1.$
\end{definition}

Every covector $\psi\in \mathfrak{g}^{\ast}$ can be considered as a vector $\psi\in\mathfrak{g},$ setting $\psi(v)=\langle \psi,v\rangle $ for
each $v\in \mathfrak{g}.$
Then the function $M(t)$ in (\ref{max}) takes the form
\begin{equation}
\label{M1}
M(t):=\langle\psi(t),u(t)\rangle=\min_{u\in U}\langle\psi(t),u\rangle=M_0\geq 0.
\end{equation}
Moreover, the equality (\ref{hame1}) takes the form 
\begin{equation}
\label{hame2}_
\langle \psi'(t),w\rangle= \langle\psi,[u(t),w]\rangle,\,w\in \mathfrak{g}.
\end{equation}

The vectors $e_0,-e_1,\dots,-e_n$ form a basis in $\mathfrak{g}^{\ast}=T_e^{\ast}G$, dual to the basis $e_0,e_1,\dots,e_n$ in $\mathfrak{g}$, because
$e_i(e_j)=\langle e_i,e_j\rangle=0$ for $i\neq j$, $i,j=0,\dots,n$,
$$e_0(e_0)=\langle e_0,e_0\rangle=1;\quad -e_i(e_i)=-\langle e_i,e_i\rangle=1,\,\,i=1,\dots,n.$$

\begin{proposition}
\label{neg}
Let $v, w\in\mathfrak{p}$ be respectively future directed timelike vector and spacelike vector.
Then the ranges of functions $f(u):=\langle v, u\rangle$ and $g(u):=\langle w, u\rangle,$ $u\in U,$ are equal to $[\|v\| ,+\infty)$ and $\mathbb{R}.$ 	
\end{proposition}

\begin{proof}
It is known that the general orthochronic Lorentz group $SO_0(1,n)$ acts transitively on each of the following sets: $U,$ $\partial U,$ 
the set of all future directed isotropic vectors and the set of all (spacelike) vectors $w\in\mathfrak{g}$ such that $\langle w, w\rangle= -\beta^2$ for a fixed $\beta >0.$
Therefore we can assume that $v=\alpha e_0,$ $w=\beta e_1,$ $\alpha>0,$ $\beta>0.$
	
It is clear that $u:={\rm ch} te_0 + {\rm sh} t e_1\in\partial U,$ $t\in \mathbb{R}$. Then  
$$\langle v, u\rangle= \alpha ({\rm ch} t)=\|v\|({\rm ch}t),\, \langle w,u\rangle= -\beta ({\rm sh} t),$$   
whence follows Proposition \ref{neg}, since $\langle v, u\rangle \geq\alpha$ for $u\in U$ in general case.
\end{proof}

Due to Proposition \ref{neg}, the equality in (\ref{M1}) holds only when the orthogonal projection of the vector 
$\psi(t)\in\mathfrak{g}$ to $\mathfrak{p}$ is equal to zero ($M_0=0$) or a future directed timelike vector in $\mathfrak{p}$ ($M_0>0$).

Let $g(t)$, $t\in\mathbb{R}$, be a normal timelike extremal and $M_0=1$. 
Let us denote by $\psi_i(t)$, $i=0,\dots,n$, the coordinates of the (future directed timelike) covector function
$\psi(t)$ in the specified dual basis.

Then 
\begin{equation}
\label{psi}
\psi(t)=\psi_0(t)e_0-\sum\limits_{i=1}^{n}\psi_i(t)e_i,\, \mbox{and besides }\, \psi_0(t)>0.
\end{equation} 

Set us similarly $u(t)=u_0(t)-\sum_{i=1}^{r}u_i(t)e_i,$ where $u_0(t)>0.$ 

Then (\ref{M1}) has the form
\begin{equation}
\label{min}
\langle\psi(t),u(t)\rangle=\psi_0(t)u_0(t)-\sum_{k=1}^{r}\psi_k(t)u_k(t)=\min_{u\in U}\langle\psi(t),u\rangle=1.
\end{equation} 
The equality (\ref{min}) is valid iff $u_k(t)=\psi_k(t),$ $k=0,\dots, r$, and $\langle u(t),u( t)\rangle=1$.
Due to the first equality in (\ref{min}), reasoning as in \cite{BerZub}, \cite{Ber14}, we obtain the theorem.

\begin{theorem}
\label{main}
Any normal timelike extremal of a left-invariant (sub-)Lorentzian (anti)metric on the Lie group $G$,
parameterized by arclength, is a solution to the system of differential equations
\begin{equation}
\label{equat}
g'(t)=dl_{g(t)}(u(t)),\quad u(t)=\psi_0(t)e_0-\sum\limits_{i=1}^{r}\psi_i(t)e_i,\quad \|u(0)\|=1,\quad\psi_0(t)>0,
\end{equation}
\begin{equation}
\label{equat1}
\psi_j'(t)=\sum\limits_{k=0}^{n}\left(C_{0j}^{k}\psi_0\psi_k-\sum\limits_{i=1}^{r}C_{ij}^{k}\psi_i\psi_k\right),\,\,j=0,\dots,n.
\end{equation}
Here $C_{ij}^{k}$ are structure constants in the basis $e_0,\dots, e_n$ of the Lie algebra $\mathfrak{g}$.
\end{theorem}

In case of an isotropic curve, there is no notion of parameterization by arclength.
Therefore, it is natural to accept the normalization $u_0(t)=\psi_0(t)\equiv 1.$ 
Set
$$W:=\{w\in\mathfrak{p}\mid \langle w,e_0\rangle=0,\,\langle w,w\rangle=-1\}.$$
Then the condition for the isotropy of the curve under the second equality in (\ref{equat}) has the form
\begin{equation}
\label{min1} 
\langle\psi(t),w(t)\rangle=\min_{w\in W}\langle\psi(t),w\rangle=-1,\,\mbox{где } w(t)=u(t)-u_0(t)e_0.
\end{equation}

The following theorem is established similarly to Theorem \ref{main}.

\begin{theorem}
\label{isotrn}
Any normal isotropic extremal of a left-invariant (sub-)Lorentzian (anti)metric on the Lie group $G$ is a solution to the system of differential equations
\begin{equation}
\label{isoequat}
g'(t)=dl_{g(t)}(u(t)),\quad u(t)=\psi_0(t)e_0-\sum\limits_{i=1}^{r}\psi_i(t)e_i,\quad \|u(0)\|=0,\quad\psi_0(t)\equiv 1,
\end{equation}
\begin{equation}
\label{isoequat1}
\psi_j'(t)=\sum\limits_{k=0}^{n}\left(C_{0j}^{k}\psi_0\psi_k-\sum\limits_{i=1}^{r}C_{ij}^{k}\psi_i\psi_k\right),\,\,j=1,\dots,n.
\end{equation}
\end{theorem}

\section{Specifying a sub-Lorentzian structure on $\GL^+(2,\mathbb{C})$}

Let us recall that $\GL^{+}(2,\mathbb{C})$ is the connected Lie group of all complex $(2\times 2)$-matrices with positive determinant.
Its Lie algebra $\mathfrak{gl}^{+}(2,\mathbb{C})$ consists of all complex $(2\times 2)$-matrices with real trace and has a basis
\begin{equation}
\label{e1}
e_0=\frac{\sigma_0}{2},\quad e_1=\frac{\sigma_1}{2},\quad e_2=\frac{\sigma_2}{2},\quad e_3=\frac{\sigma_3}{2},
\end{equation}
where
\begin{equation}
\label{sigma}
\sigma_0=\left(\begin{array}{cc}
1 & 0  \\
0 & 1
\end{array}\right),\quad \sigma_1=\left(\begin{array}{cc}
0 & 1  \\
1 & 0
\end{array}\right),\quad \sigma_2=\left(\begin{array}{cc}
0 & {\bf i}  \\
-{\bf i} & 0
\end{array}\right),\quad \sigma_3=\left(\begin{array}{cc}
1 & 0  \\
0 & -1
\end{array}\right),
\end{equation}

\begin{equation}
\label{e2}
e_4={\bf i}e_1,\quad e_5={\bf i}e_2,\quad e_6={\bf i}e_3.
\end{equation}

The matrices (\ref{sigma}) are Hermitian (the last three of them are called {\it Pauli matrices}) and, like the matrices (\ref{e1}), constitute the real basis of the linear space $H$ of all Hermitian complex $(2\times 2)$-matrices.

The matrices $e_4, e_5, e_6$ are skew-Hermitian.

Let us write down all non-zero commutation relations:
\begin{equation}
\label{b1}
[e_4,e_5]=-[e_1,e_2]={\bf i}[e_4,e_2]={\bf i}[e_1,e_5]=e_6;
\end{equation}
\begin{equation}
\label{b2}
[e_5,e_6]=-[e_2,e_3]={\bf i}[e_5,e_3]={\bf i}[e_2,e_6]=e_4;
\end{equation}
\begin{equation}
\label{b3}
[e_6,e_4]= -[e_3,e_1]={\bf i}[e_6,e_1]={\bf i}[e_3,e_4]=e_5.
\end{equation}
Due to the last two equalities in each of the formulas (\ref{b1}), (\ref{b2}), (\ref{b3}),
\begin{equation}
\label{b4}
-[e_2,e_4]=[e_1,e_5]=e_3,\quad -[e_3,e_5]=[e_2,e_6]=e_1,\quad -[e_1,e_6]=[e_3,e_4]=e_2.
\end{equation}

Due to (\ref{b1}), (\ref{b2}), (\ref{b3}), the skew-Hermitian matrices $e_5,e_6,e_7$ form the basis of some Lie algebra. This is the Lie algebra $\mathfrak{su}(2)$,
consisting of all skew-Hermitian complex $(2\times 2)$-matrices with zero trace:
$$\mathfrak{su}(2)=\left\{\left(\begin{array}{cc}
{\bf i}X & Y \\
-\overline{Y} & -{\bf i}X
\end{array}\right)\left|\,\,X\in\mathbb{R},\,\,Y\in\mathbb{C}\right.\right\}.$$
Let us recall that $\mathfrak{su}(2)$ is the Lie algebra of the compact simply connected Lie group $\SU(2)$ of
all unitary unimodular $(2\times 2)$-matrices.

\begin{remark}
\label{rem}
It follows from  (\ref{b1}) -- (\ref{b4}) that
$\mathfrak{gl}^{+}(2,\mathbb{C})=\mathfrak{su}(2)\oplus H$ and
\begin{equation}
\label{kp}
[\mathfrak{su}(2),\mathfrak{su}(2)]=\mathfrak{su}(2),\quad [\mathfrak{su}(2),H]\subset H,\quad [H,H]=\mathfrak{su}(2).
\end{equation}
As a consequence, $H$ generates the Lie algebra $\mathfrak{gl}^{+}(2,\mathbb{C})$.
Moreover, $\mathfrak{su}(2)$ is a maximal compact subalgebra of the Lie algebra $\mathfrak{gl}^{+}(2,\mathbb{C})$.
\end{remark}

Let us define the Lorentzian quadratic form on $\mathfrak{gl}^{+}(2,\mathbb{C})$:
$$\langle u,u\rangle=u_0^2-\sum\limits_{k=1}^{6}u_k^2,\quad\mbox{where }\, u=\sum\limits_{k=0}^{7}u_ke_k\in\mathfrak{gl}^{+}(2,\mathbb{C}).$$
The corresponding pseudoscalar product on $\mathfrak{gl}^{+}(2,\mathbb{C})$ is
$$\langle u_1,u_2\rangle=\frac{1}{2}\left(\langle u_1+u_2,u_1+u_2\rangle-\langle u_1,u_1\rangle-\langle u_2,u_2\rangle\right).$$
Note that the vectors $e_0,e_1,\dots,e_6$ form an orthonormal basis of the Lie algebra $\mathfrak{gl}^{+}(2,\mathbb{C})$, since
$$\langle e_0,e_0\rangle=1,\quad \langle e_i,e_i\rangle=-1,\,\,i=1,\dots,6;\quad \langle e_i,e_j\rangle=0\quad \mbox{for }i\neq j,\,\,i,j=0,\dots,6.$$
In this case, the vectors $e_0,e_1,e_2,e_3$ form an orthonormal basis of the vector space $H$ with respect to the pseudoscalar product $\langle\cdot,\cdot\rangle$ induced from $\mathfrak{gl}^{+}(2,\mathbb{C})$, and
$$\langle h,h\rangle=h_0^2-h_1^2-h_2^2-h_3^2=4 {\rm det}h,\quad\mbox{where }\, h=\sum\limits_{i=0}^{3}h_ie_i\in H.$$

The pair $(\mathfrak{p}:=H,\langle\cdot,\cdot\rangle)$ with the vector $v:=e_0$ defines a left-invariant sub-Lorentzian structure on $\GL^{+}(2,\mathbb{C })$.

\section{Sub-Riemannian geodesics on $\SL(2,\mathbb{C})$}
\label{subriem}

Let us remind that $\SL(2,\mathbb{C})$ is a subgroup of the group $\GL^{+}(2,\mathbb{C})$, consisting of all complex $(2\times 2)$- matrices with determinant equal to $1$. The vectors $e_i,$ $i=1,\dots,6$, given by the formulas (\ref{e1})\,--\,(\ref{e2}), constitute the orthonormal basis of its Lie algebra $\mathfrak{sl}(2,\mathbb{C})$ with the scalar product $(\cdot,\cdot)$.
 
It follows from (\ref{b1}) -- (\ref{b4}) that for $\mathfrak{p}_0=H_0:=H\cap \mathfrak{sl}(2,\mathbb{C}),$
\begin{equation}
\label{kp}	
\mathfrak{sl}(2,\mathbb{C})=\mathfrak{su}(2)\oplus \mathfrak{p}_0,\quad [\mathfrak{su}(2),\mathfrak{p}_0]= \mathfrak{p}_0,\quad [\mathfrak{p}_0,\mathfrak{p}_0]=\mathfrak{su}(2).
\end{equation}
As a consequence, $\mathfrak{p}_0=H_0$ generates the Lie algebra $\mathfrak{sl}(2,\mathbb{C})$. 

The pair $(H_0,(\cdot,\cdot))$ defines a left-invariant sub-Riemannian metric $\rho$ on $\SL(2,\mathbb{C})$: 
$\rho(x,y)$ is the infimum of the lengths of piecewise continuously differentiable horizontal paths $\gamma=\gamma(t)$, $0\leq t\leq a$,
joining $x,y\in \SL(2,\mathbb{C})$, i.e., the paths with conditions
$$\gamma(0)=x,\,\,\gamma(a)=y,\quad \gamma'(t)=dl_{\gamma(t)}(u(t)),\quad u(t)\in H_0,\,\,(u(t),u(t))\leq 1,$$
the length of each such path is determined by the formula
$l(\gamma)=\int\limits_{0}^{a}\sqrt{(u(t),u(t))}dt.$

In Theorems~3, 4 from \cite{Ber18}, were proved general results on Riemannian symmetric spaces, their isometry and isotropy groups, their Lie algebras satisfying analogues to the relations (\ref{kp}), and sub-Riemannian metrics on isometry groups defined by scalar products on the subspace $\mathfrak{p}_0;$
similar results in somewhat less general cases were obtained earlier in \cite{BCHG02}. From these theorems, applied to the Riemannian symmetric space
$\SL(2,\mathbb{C})/\SU(2)=L^3,$ it follows that for the sub-Riemannian space $(\SL(2,\mathbb{C}),\rho)$ holds the following proposition.

\begin{proposition}
\label{subrgeod}
Each parametrized by arclength geodesic $\gamma=\gamma(t)$, $t\in\mathbb{R}$, in $(\SL(2,\mathbb{C}),\rho)$ with condition $\gamma(0)=e$ is normal and it is a product of two 1-parameter subgroups:
\begin{equation}
\label{geod2}
\gamma(t)=\exp\left(t\sum\limits_{i=1}^{6}\alpha_ie_i\right)\exp\left(-t\sum\limits_{i=4}^6\alpha_ie_i\right),
\end{equation}
where $\alpha_i$, $i=1,\dots,6$, are some arbitrary constants such that
\begin{equation}
\label{norm2}
\alpha_1^2+\alpha_2^2+\alpha_3^2=1.
\end{equation}
\end{proposition}

\begin{theorem}
\label{ppp}
Each segment of a 1-parameter subgroup
\begin{equation} 
\label{onepar}
\gamma(t)=\exp(t X),\,\, t\in\mathbb{R},\,\, X\in \mathfrak{p}_0=H_0,\,\, (X,X)=1,
\end{equation}
is a shortest arc of the sub-Riemannian space $(\SL(2,\mathbb{C}),\rho)$.
\end{theorem}

\begin{proof}
First of all, note that each segment of the curve (\ref{onepar}) has the same length with respect to $\rho$ and the left-invariant
Riemannian metric $\rho_1\leq \rho$ on $\SL(2,\mathbb{C}),$ defined by the scalar product $(\cdot,\cdot)$ on the Lie algebra $\mathfrak{sl}(2,\mathbb{C})$
of the Lie group $\SL(2,\mathbb{C}).$
By Proposition \ref{isom1}, the scalar product $(\cdot,\cdot)$ in $\mathfrak{sl}(2,\mathbb{C})$ and the direct sum orthogonal to $(\cdot,\cdot)$ in (\ref{kp}) are $\Ad(\SU(2))$-invariant.
Therefore the canonical projection
$$\pr: (\SL(2,\mathbb{C}),\rho_1)\rightarrow (\SL(2,\mathbb{C})/\SU(2)=L^3,\rho_2)$$
onto the Riemannian symmetric space $(L^3,\rho_2)$ is a Riemannian submersion, where $\rho_2$ is an invariant Riemannian metric of constant sectional
curvature $K$ on $L^3,$ uniquely determined by the metric $\rho_1$.

As in general case of Riemannian symmetric spaces, $\pr$ maps isometrically a 1-parameter subgroup (\ref{onepar}), tangent to the horizontal distribution
of Riemannian submersion $\pr,$ onto a geodesic in $(L^3,\rho_2).$	
	
Let us clarify that due to the last equality in (\ref{kp}), for each
$Z\in\mathfrak{p}_0$ we have $[X,Z]\in \mathfrak{su}(2).$
Therefore, the projection of this vector onto $\mathfrak{p}_0$ with respect to the expansion in (\ref{kp}) is equal to $[X,Z]_{\mathfrak{p}_0}=0$.
Thus, we have $([X,Z]_{\mathfrak{p}_0},X)=0$ for each $Z\in\mathfrak{p}_0,$ i.e., the condition 3) of Theorem 5.1.2 from
\cite{BerNik20} holds.
	
Due to the theorem, $\pr(\gamma(t)),$ $t\in\mathbb{R},$ is a geodesic in $(L^3,\rho_2).$
	
Moreover, $\pr$ does not increase distances in general.
	
Theorem \ref{ppp} follows from the above.
\end{proof}

\begin{remark}
\label{curv}
According to Theorems 12.2 of Chapter II and 4.2 of Chapter IV in \cite{Hel}, the constant sectional curvature $K$ can be calculated using the formula $K=([[e_1,e_2],e_1],e_2).$
Calculations using formulas (\ref{b1})---(\ref{b4}) give $K=-1.$
\end{remark}	

Two corollaries follow from the proof of Theorem \ref{ppp}.

\begin{corollary}
\cite{BerGui00}.
\label{subm}	
$\pr: (\SL(2,\mathbb{C}),\rho)\rightarrow (\SL(2,\mathbb{C})/\SU(2)=L^3,\rho_2)$ is a submetry.
\end{corollary}

\begin{corollary}
\label{diff}	
The mapping $\exp: H_0\rightarrow \exp(H_0)$ is a diffeomormism.
\end{corollary}

\section{Applications of sub-Riemannian geodesics on $\SL(2,\mathbb{C})$}
\label{apple}

It is easy to see that the following proposition is true.

\begin{proposition}
\label{iso}
$\GL^+(2,\mathbb{C})$ is isomorphic to the direct product $\mathbb{R}_+I\times\SL(2,\mathbb{C})$, and $\GL(2,\mathbb{C})$ is isomorphic to the
direct product $\mathbb{R}_+I\times S^1\cdot \SL(2,\mathbb{C}),$ where $R_{+}$ is the multiplicative group of positive real numbers.
\end{proposition}

\begin{theorem}
\label{slgeod}
Let $g\in\GL^+(2,\mathbb{C})$ and
\begin{equation}
\label{cond0} 
g=e^{\xi/2}g_1,\,\mbox{where } g_1\in \SL(2,\mathbb{C}),\,\, \xi\geq \eta:=\rho(e,g_1)>0,
\end{equation} 
$\gamma=\gamma(t),$ $0\leq t\leq \eta,$ be an arbitrary shortest arc in $(\SL(2,\mathbb{C}),\rho)$, parametrized by arclength, such that $\gamma(0)=e,$ $\gamma(\eta)=g_1.$
	
If $\xi>\eta$ then the curve
$$g(t)=e^{{\rm ch(c)}t/2}\gamma({\rm sh(c)}t),\,\, 0\leq t\leq k,\quad\mbox{где }k>0,\,\,c>0,\,\,\xi=k{\rm ch}(c),\,\,\eta=k{\rm sh}(c),$$
is a timelike longest arc in $(\GL^+(2,\mathbb{C}),d),$ which joins $e$ and $g$.
	
If $\xi=\eta$ then the curve $g(t)=e^{t/2}\gamma(t),$ $0\leq t\leq \xi,$ is an isotropic longest arc in $(\GL^+(2,\mathbb{C}),d)$ joining $e$ and $g$.

If $g\notin \mathbb{R}_+I$ and the condition (\ref{cond0}) is not satisfied, then there is no future directed nonspacelike horizontal
curve joining $e$ and $g.$	
\end{theorem}

\begin{proof}
We get $g'(t)=dl_{g(t)}({\rm ch}(c)e_0+{\rm sh}(c)u({\rm sh}(c)t)),$ where 
$$u({\rm sh}(c)t)=\sum_{i=1}^3u_i({\rm sh}(c)t)e_i,\quad (u({\rm sh}(c)t),u({\rm sh}(c)t))=1.$$
Hence, $g(t),$ $0\leq t\leq k,$ is a  future directed timelike curve, and its length is equal to 
\begin{equation}
\label{lg}	
L=k=\int_0^k \sqrt{{\rm ch}^2(c)-{\rm sh}^2(c)(u({\rm sh}(c)t),u({\rm sh}(c)t))}dt.
\end{equation} 
Moreover, the length $l$ of the shortest arc $\gamma({\rm sh}(c)t),$ $0\leq t\leq k,$ in $(\SL(2,\mathbb{C}),\rho),$ parametrized proportionally to the arclength
with a factor ${\rm sh}(c)>0,$ is equal to the minimal length of horizontal arcs in $\SL(2,\mathbb{C})$ joining $e$ and $g_1$:
$$l=\int_0^k\sqrt{{\rm sh}^2(c)(u({\rm sh}(c)t),u({\rm sh}(c)t))}\equiv k{\rm sh}(c).$$
Due to (\ref{lg}), $g(t),$ $t\in [0,k],$ is a 
future directed timelike longest arc in $(\GL^+(2,\mathbb{C}),d)$, parametrized by arlength and joining $e$, 
$g.$

If $\xi=\eta$ then replacing ${\rm ch}(c)$ with $1$ in the above argument,
we find that $g(t),$ $0\leq t \leq\xi,$ is an isotropic longest arc in $(\GL^+(2,\mathbb{C}),d).$
	
To prove the last statement, we can assume that
\begin{equation}
\label{cond1} 
g=e^{\xi/2}g_1,\,\,\mbox{where } g_1\in \SL(2,\mathbb{C}),\,\, 0< \xi < \eta:=\rho(e,g_1).
\end{equation}
Suppose that there is a future directed nonspacelike horizontal (Lipschitzian)
curve joining $e$ and $g.$ We can choose a parameterization of this curve such that it has the form
$g(t)=e^{t/2}\gamma(t),$ $0\leq t\leq \xi,$ where $\gamma(t),$ $0\leq t\leq \xi,$ is a Lipschitzian curve in $(\SL(2,\mathbb{C}),d)$.

Then $g'(t)=dl_{g(t)}(e_0+v(t)),$ where 
$v(t)=\sum_{i=1}^3v_i(t)e_i$ is a measurable vector function, defined almost everywhere, with a bounded from above scalar square $(v(t),v(t)).$
Since the curve is nonspacelike, we have almost everywhere 
$$\langle e_0+v(t), e_0+v(t)\rangle = 1-(v(t),v(t))\geq 0.$$ 
Then the inequality $(v(t),v(t))\leq 1$ holds almost everywhere. 
But we have $\gamma'(t)=dl_{\gamma(t)}(v(t))$ for almost all $t\in [0,\xi].$
Therefore, the length of the curve $\gamma(t),$ $0\leq t\leq \xi,$ joining the points $e$ and $g_1$ in $(\SL(2,\mathbb{C}),\rho),$ is equal to $l=\int_0^{\xi}\sqrt{(v(t),v(t))}dt\leq \xi.$
This contradicts the inequalities in (\ref{cond1}).  	
\end{proof}

\begin{corollary}
\label{cor}	
Every future directed nonspacelike curve $e^{t/2}\gamma(t),$ $t\in\mathbb{R}$, where $\gamma(t),$ $t\in\mathbb{R},$ is
a geodesic in $(\SL(2,\mathbb{C}),\rho)$ parametrized proportionally to the arclength, is a geodesic in $(\GL^+(2,\mathbb{C}),d),$ parametrized proportionally to the arclength for a timelike curve.
\end{corollary}		

The following proposition is proved in a similar way to Theorem \ref{slgeod}.

\begin{proposition}
\label{slpr}
Any segment of 1-parameter subgroup $g(t)=2e^{t/2}e_0,$ parametrized by arclength, is a timelike longest arc in the space $(\GL^+( 2,\mathbb{C}),d).$
\end{proposition}

\section{Sub-Lorentzian normal nonspacelike extremals on $\GL^{+}(2,\mathbb{C})$}

In Theorem \ref{slgeod} and Proposition \ref{slpr}, were precisely established the elements $g$ of the Lie group $\GL^+(2,\mathbb{C}),$ for which there exists a future directed nonspacelike sub-Lorentzian longest arc with the origin  $e$ and the endpoint $g$.

Moreover, the length $d(e,g)$ of the longest arc is found. 

Therefore, we can apply Theorem \ref{main} to search for normal timelike extremals and thereby geodesics and possibly the other longest arcs.

\begin{theorem}
\label{main2}
Each normal timelike extremal $g(t)$, $t\in\mathbb{R}$, of sub-Lorentzian (anti)metric $d$ on
$(\GL^{+}(2,\mathbb{C}),d)$  with origin $g(0)=e$ is a product of two 1-parameter
subgroups:
\begin{equation}
\label{geod}
g(t)=\exp\left(t\sum\limits_{i=0}^6\alpha_ie_i\right)\exp\left(-t\sum\limits_{i=4}^6\alpha_{i}e_i\right),
\end{equation}
where $\alpha_i$, $i=0,\dots,6$, are some arbitrary constants such that
\begin{equation}
\label{norm}
\alpha_0=\sqrt{1+\alpha_1^2+\alpha_2^2+\alpha_3^2}.
\end{equation}
\end{theorem}

\begin{proof}
Due to Theorem \ref{main}, the considered extremal is a solution to the system (\ref{equat}), (\ref{equat1}) for $n=6,$ $r=3$.	
Set 
$$\psi_0(0)=\alpha_0,\, \psi_i(0)=-\alpha_i,\, 1,\dots,6,\, v(t)=u(t)-u_0(t)e_0=\psi(t)-\psi_0(t)e_0:=\varphi(t).$$	
It follows from Proposition \ref{iso} that $C^k_{0j}=0,$ $C^0_{ij}=0,$ $C^k_{i0}=0.$ Then $\psi_0\equiv \alpha_0,$ 
\begin{equation}
\label{equat2}
\psi_j'(t)=-\sum\limits_{k=1}^{6}\sum\limits_{i=1}^{3}C_{ij}^{k}\psi_i\psi_k,\,\,j=1,\dots,6. 
\end{equation}
Now (\ref{min}) and (\ref{equat}) imply (\ref{norm}), $\psi_0(t)\equiv \alpha_0=\rm{ch}c$,
\begin{equation}
\label{v}	
\langle v(t),v(t)\rangle=-(v(t),v(t))=-(\varphi(t),\varphi(t))\equiv -(\alpha_1^2+\alpha_2^2+\alpha_3^2)=-\rm{sh}^2c,\, c\geq 0.
\end{equation}

If $c=0$ then $u_0(t)=\psi_0(t)\equiv\alpha_0=1,$ $u_i(t)=\psi_i(t)\equiv -\alpha_i=0,$ $i=1,2,3,$ and due to the first formula in (\ref{equat}),
$g(t)=2e^{t/2}e_0$ has the form (\ref{geod}).

Let $c>0.$ Set
\begin{equation}
\label{set}	
\tilde{v}(t)=-v(t),\, \widetilde{\psi}_j(t)=-\psi_j(t),\, j=1,\dots,6;\, \tilde{u}_i(t)=\widetilde{\psi}_i(t),\, i=1,2,3.
\end{equation}
Then 
\begin{equation}
\label{equat3}
\widetilde{\psi}_j'(t)=\sum\limits_{k=1}^{6}\sum\limits_{i=1}^{3} C_{ij}^{k}\widetilde{\psi}_i\widetilde{\psi}_k,\,\,j=1,\dots,6;\, \tilde{v}(t)=\sum_{i=1}^3\tilde{u}_i(t)e_i,\, (\tilde{v}(t),\tilde{v}(t))=\rm{sh}^2c. 
\end{equation}	
Due to (\ref{set}), (\ref{equat3}), there is a normal sub-Riemannian extremal--geodesic $\gamma(t),$ $t\in\mathbb{R},$ such that $\gamma(0)=e,$ $\gamma'(t)=dl_{\gamma(t)}(\tilde{v}(t)),$ $t\in\mathbb{R},$ in $(\SL( 2,\mathbb{C}),\rho)$, parametrized proportionally to the arclength with the multiplier $\rm{sh}c.$ As a corollary of the above, Propositions \ref{subrgeod} and \ref{iso}, we get
$$g(t)=2e^{\rm{ch}(c)/2}e_0\gamma(t)=2e^{\rm{ch}(c)/2}e_0\exp\left(t\sum\limits_{i=1}^{6}\alpha_ie_i\right)\exp\left(-t\sum\limits_{i=4}^6\alpha_ie_i\right)=$$
$$\exp\left(t\sum\limits_{i=0}^{6}\alpha_ie_i\right)\exp\left(-t\sum\limits_{i=4}^6\alpha_ie_i\right),\, t\in \mathbb{R}.$$
\end{proof}

Taking into account Theorem \ref{isotrn}, we prove similarly to Theorem \ref{main2} the following

\begin{theorem}
\label{main3}
Each normal isotropic extremal $g(t)$, $t\in\mathbb{R}$, of sub-Lorentzian (anti)me\-tric $d$ on
$\GL^{+}(2,\mathbb{C})$ with origin  $g(0)=e$ is a product of two 1-parameter
subgroups:
$$g(t)=\exp\left(t\sum\limits_{i=0}^6\alpha_ie_i\right)\exp\left(-t\sum\limits_{i=4}^6\alpha_ie_i\right),$$
where $\alpha_i$, $i=0,\dots,6$,  are some arbitrary constants such that
\begin{equation}
\label{norm1}
\alpha_0=\sqrt{\alpha_1^2+\alpha_2^2+\alpha_3^2}=1.
\end{equation}
\end{theorem}

\begin{proposition}
If the vector $\{\alpha_4,\alpha_5,\alpha_6\}$ is collinear to the vector $\{\alpha_1,\alpha_2,\alpha_3\}$, 
then the timelike (and isotropic) normal extremal $g(t)$ given by the formula (\ref{geod}), is a 1-parameter subgroup:
$$g(t)=\exp\left(t\sum\limits_{i=0}^3\alpha_ie_i\right).$$
\end{proposition}

\begin{proof}
Let $\alpha_4=a\alpha_1$, $\alpha_5=a\alpha_2$, $\alpha_6=a\alpha_3$ for some $a\in\mathbb{R}$.
Due to the form of the vector $e_0$, (\ref{e2}) and (\ref{geod}),
$$g(t)=e^{\alpha_0t/2}\exp\left(t(1+{\bf i}a)(\alpha_1e_1+\alpha_2e_2+\alpha_3e_3)\right)\exp\left(-{\bf i}ta(\alpha_1e_1+\alpha_2e_2+\alpha_3e_3)\right)$$
$$=e^{\alpha_0t/2}\exp(t(\alpha_1e_1+\alpha_2e_2+\alpha_3e_3))=\exp\left(t\sum\limits_{i=0}^3\alpha_ie_i\right).$$
\end{proof}

\begin{proposition}
\label{expon}
Set $A=\sum\limits_{i=0}^{3}\alpha_ie_i$, 
$\alpha_0,\dots,\alpha_3\in\mathbb{C}$, $w=\frac{1}{2}\sqrt{\alpha_1^2+\alpha_2^2+\alpha_3^2}$.
Then 
$$\exp(tA)=e^{\alpha_0 t/2}\left(2e_0+t\sum\limits_{i=1}^{3}\alpha_ie_i\right),\quad\mbox{if } w=0,$$ 
$$\exp(tA)=e^{\alpha_0t/2}\left(2{\rm ch}(wt)e_0+\frac{{\rm sh}(wt)}{w}\sum\limits_{i=1}^{3}\alpha_ie_i\right),\quad\mbox{if } w\neq 0.$$
\end{proposition}

\begin{corollary}
\label{form}
Set 
$$w_1=\frac{1}{2}\sqrt{(\alpha_1+{\bf i}\alpha_4)^2+(\alpha_2+{\bf i}\alpha_5)^2+(\alpha_3+{\bf i}\alpha_6)^2},\quad w_2=\frac{1}{2}\sqrt{\alpha_4^2+\alpha_5^2+\alpha_6^2};$$
$$m_1={\rm ch}(w_1t),\quad n_1=\left\{\begin{array}{cc}
{\rm sh}(w_1t)/w_1, & \mbox{if }\, w_1\neq 0, \\
t, & \mbox{if }\, w_1=0,
\end{array}\right.;$$
$$m_2=\cos(w_2t),\quad  n_2=\left\{\begin{array}{cc}
{\rm sin}(w_2t)/w_2, & \mbox{if }\, w_2\neq 0, \\
t, & \mbox{if }\, w_2=0,
\end{array}\right..$$
Then the normal nonspacelike, parametrized by arclength in the timelike case,  extremal $g(t)$, $t\in\mathbb{R}$, of
sub-Lorentzian (anti)metric $d$ on
$\GL^{+}(2,\mathbb{C})$ (see Theorems \ref{main2}, \ref{main3}) is equal to
$$g(t)=\sum\limits_{i=0}^{6}c_i(t)e_i+c_7(t){\bf i}e_0,$$
where
$$c_0(t)=2e^{\alpha_0t/2}(m_1m_2+n_1n_2w_2^2),\quad c_7(t)=-\frac{e^{\alpha_0t/2}}{2}n_1n_2(\alpha_1\alpha_4+\alpha_2\alpha_5+\alpha_3\alpha_6),$$
$$c_1(t)=\frac{e^{\alpha_0t/2}}{2}n_1(2\alpha_1m_2+(\alpha_3\alpha_5-\alpha_2\alpha_6)n_2),\quad c_2(t)=\frac{e^{\alpha_0t/2}}{2}n_1(2\alpha_2m_2+(\alpha_1\alpha_6-\alpha_3\alpha_4)n_2),$$
$$c_3(t)=\frac{e^{\alpha_0t/2}}{2}n_1(2\alpha_3m_2+(\alpha_2\alpha_4-\alpha_1\alpha_5)n_2),\quad
c_i(t)=e^{\alpha_0t/2}(m_2n_1-m_1n_2)\alpha_i,\quad i=4,5,6.$$
\end{corollary}

\begin{proof}
Due to (\ref{e2}) and Proposition \ref{expon},
$$\exp\left(t\sum\limits_{i=0}^6\alpha_ie_i\right)=e^{\alpha_0t/2}\left(2m_1e_0+n_1\sum\limits_{i=1}^{6}\alpha_ie_i\right);$$
\begin{equation}
\label{hh}
\exp\left(-t\sum\limits_{i=4}^6\alpha_ie_i\right)=\exp\left(t\sum\limits_{i=4}^6(-{\bf i}\alpha_i)e_{i-3}\right)
=2m_2e_0-n_2\sum\limits_{i=4}^6\alpha_ie_i.
\end{equation}
On the base of Theorems \ref{main2}, \ref{main3}, it remains to multiply these matrix exponents, taking into account (\ref{e2}) and the equalities
$e_i^2=\frac{e_0}{2},$ $i=0,1,2,3$,
\begin{equation}
\label{ll}
e_1\cdot e_2=-e_2\cdot e_1=-\frac{e_6}{2},\quad e_1\cdot e_3=-e_3\cdot e_1=\frac{e_5}{2},\quad e_2\cdot e_3=-e_3\cdot e_2=-\frac{e_4}{2}.
\end{equation}
\end{proof}

\begin{proposition}
\label{isom}
The action of the Lie subgroup $\SU(2)\subset\GL^+(2,\mathbb{C})$ on the Lie group $\GL^+(2,\mathbb{C})$ by conjugations is a subgroup of internal automorphisms of the group
$\GL^+(2,\mathbb{C})$ and defines an isometry subgroup with respect to the left-invariant sub-Lorentzian (anti)metric $d$ on $\GL^+(2,\mathbb{C})$
defined by the pair $(H,\langle\cdot,\cdot\rangle).$
Adjoint action $\Ad(\SU(2))$ of the group $\SU(2)$ on the Lie algebra $\mathfrak{gl}^+(2,\mathbb{C})$ defines
an automorphism subgroup of this algebra, transforming $H$ and $\mathfrak{su}(2)$ into itself, and an isometry subgroup of the space $(H,\langle\cdot,\cdot\rangle),$ acting transitively on $\{w\in H_0\mid \langle w,w\rangle=-1\}.$
\end{proposition}

\begin{proof}
The indicated action by conjugations is determined by the formula	
$$s\in\SU(2)\rightarrow \Inn(s): g\in\GL^+(2,\mathbb{C})\rightarrow sgs^{-1}$$
and defines a subgroup of the automorphism group of the Lie group $\GL^+(2,\mathbb{C}).$
$\Ad(\SU(2))$ is a subgroup of the automorphism group of the Lie algebra $\mathfrak{gl}^+(2,\mathbb{C})$, and since $\GL^+(2,\mathbb {C})$ is a matrix group, its action is determined by the same formula
\begin{equation}
\label{Ad}
\Ad(s)(\cdot)=d\Inn(s)_e(\cdot)= s(\cdot)s^{-1}=s(\cdot)s^{\ast}.
\end{equation}
On the ground of (\ref{kp}) and the equality $\exp\circ \ad=\Ad\circ \exp$, we obtain $\Ad(\SU(2))(H)=H$, and due to (\ref{Ad}),
the restriction $\Ad(s)=l(s),$ $s\in\SU(2),$ where $l:\SL(2,\mathbb{C})\rightarrow \SO_0(1,3)$ is a mentioned in the introduction two-sheeted covering
from Theorem 2 in \cite{Ber23}, defined by the formula $l(g)(h)=ghg^{\ast},$ $g\in\SL(2,\mathbb{C}),$ $h\in (H,\langle\cdot,\cdot\rangle)$.

All statements of Proposition \ref{isom} follow from here, except for the statement about transitivity of the action.
	
In addition, $\Ad(s)(e_0)=e_0$ for all $s\in\SU(2).$ Therefore, $\Ad(\SU(2))(H_0)=H_0$ and due to equalities (\ref{b4}) and $\exp\circ\ad=\Ad\circ\exp$,
$\Ad(\SU(2))$ acts transitively on $\{w\in H_0: \langle w,w\rangle=-1\}.$
\end{proof}

On the ground of Proposition \ref{isom}, Corollary \ref{form} reduces to the case when at most one of the numbers
$\alpha_1,$ $\alpha_2,$ $\alpha_3$ is not equal to $0$ and is positive, if we exclude the trivial case when all these numbers are equal to zero.
For example, if $\alpha_1 > 0,$ $\alpha_2=0,$ $\alpha_3=0,$ then in
Corollary \ref{form} we obtain simpler formulas
$$w_1=\frac{1}{2}\sqrt{1+2{\bf i}\alpha_1\alpha_4-4w^2_2},$$
$$c_0(t)=2e^{\alpha_0t/2}(m_1m_2+n_1n_2w_2^2),\quad c_7(t)=-\frac{e^{\alpha_0t/2}}{2}n_1n_2\alpha_1\alpha_4,$$
$$c_1(t)=e^{\alpha_0t/2}n_1m_2\alpha_1,\quad c_2(t)=\frac{e^{\alpha_0t/2}}{2}n_1n_2\alpha_1\alpha_6,$$
$$c_3(t)=\frac{-e^{\alpha_0t/2}}{2}n_1n_2\alpha_1\alpha_5,\quad
c_i(t)=e^{\alpha_0t/2}(m_2n_1-m_1n_2)\alpha_i,\quad i=4,5,6.$$

\begin{theorem}
\label{slgeod1}
Normal sub-Lorentzian nonspacelike extremals in $(\GL^+(2,\mathbb{C}),d)$ with the origin $e$ are geodesics having the form of a 1-parameter subgroup from Proposition \ref{slpr} or, up to proportional reparamterization, such as in Corollary \ref{cor}.
In the second case, a segment of a geodesic $g(t),$ $0\leq t\leq t_0,$ 
is the longest arc if and only if the segment $\gamma(t),$ $0\leq t\leq t_0,$ is the shortest arc in $(\SL (2,\mathbb{C}),(\cdot,\cdot)).$

Each segment of a normal sub-Lorentzian extremal in $(\GL^+(2,\mathbb{C}),d)$ with the origin $e$, which is the longest arc, is the same as in Theorem
\ref{slgeod} or Proposition \ref{slpr}.
\end{theorem}

\begin{proof}
The proof uses Proposition \ref{iso}.
	
If $\vec{\alpha}=\{\alpha_1,\alpha_2,\alpha_3\}=\vec{0}$ then, due to Theorem \ref{main2}, we obtain the 1-parameter subgroup from Proposition \ref{slpr},
and, as a consequence of Proposition \ref{slpr}, a geodesic, each segment of which is the longest arc.	

If $\vec{\alpha}\neq\vec{0}$ then, due to Theorems \ref{main2}, \ref{main3} and formula (\ref{geod2}), the normal extremal has the form	
$g(s)=e^{\alpha_{0}s/2}\gamma(s),$ $s\in\mathbb{R},$ where $g(s)$ is parametrized by arclength, and $\gamma(s),$ $s\in\mathbb{R},$ is a geodesic in $(\SL(2,\mathbb{C}),\rho),$ parametrized proportionally to the arclength. 
Consequently, up to proportional reparametrization, we obtain a curve as in Corollary \ref{cor}, and on the base of this Corollary,  we obtain a nonspacelike geodesic. The penultimate statement is proved in exactly the same way as Theorem \ref{slgeod}.
	
The last statement of Theorem \ref{slgeod1} follows from the proven statements.
\end{proof}

The next theorem follows from Theorems \ref{ppp}, \ref{slgeod} and Proposition \ref{slpr}.	

\begin{theorem}
\label{onepa}
Each segment of a future directed nonspacelike 1-parameter subgroup $e^{t/2}\gamma(t),$ $t\in\mathbb{R}$, where $\gamma'(0)\in H_0,$ is a longest arc in $(\GL^+(2,\mathbb{C}),d)$ parametrized proportionally to the arclength for the timelike 1-parameter subgroup.
\end{theorem}

\section{Addition on sub-Riemannian geodesics on $\SL(2,\mathbb{C})$}
\label{subriem1}

The next proposition is proved in exactly the same way as Proposition \ref{isom}.

\begin{proposition}
\label{isom1}
The action of the Lie subgroup $\SU(2)\subset\SL(2,\mathbb{C})$ on $\SL(2,\mathbb{C})$ by conjugations is a subgroup of the internal automorphism group 
of the group $\SL(2,\mathbb{C})$; it preserves the subsets $\SU(2)$ and $\exp(H_0)$, and defines an isometry subgroup of the sub-Riemannian space $(\SL(2,\mathbb{C}),\rho)$.
The adjoint action $\Ad(\SU(2))$ of the group $\SU(2)$ on the Lie algebra $\mathfrak{sl}(2,\mathbb{C})$ defines a subgroup of its automorphism group and an isometry subgroup of the space $(H_0,(\cdot,\cdot)),$ acting transitively on the unit sphere $\{w\in H_0\mid (w,w)=+1\}.$
\end{proposition}

Let $\vec{\alpha}=\{\alpha_1,\alpha_2,\alpha_3\}$ and $\vec{\beta}=\{\alpha_4,\alpha_5,\alpha_6\}$ be vectors in Euclidean space $\mathbb{R}^3$, $\beta:=|\vec{\beta}|$. Let  $\gamma(\vec{\alpha},\vec{\beta};t)$, $t\in\mathbb{R}$, denote the geodesic $\gamma(t)$ given by the formula (\ref{geod2}).

Corollary \ref{form} and comparison of equalities (\ref{geod}) and (\ref{geod2}) give

\begin{proposition}
\label{srgeod}
The geodesic $\gamma(\vec{\alpha},\vec{\beta};t)$ coincides with $g(t)$ from Corollary \ref{form} for $\alpha_0=0$ and (\ref{norm2}). In addition,
$w_1=(1/2)\sqrt{1-\beta^2+2{\bf i}(\vec{\alpha}\cdot\vec{\beta})}.$
\end{proposition}

\begin{remark}
If the vector $\vec{\beta}$ is collinear to the vector $\vec{\alpha}$, then the geodesic $\gamma(\vec{\alpha},\vec{\beta};t)$ is a 1-parameter subgroup: $\gamma(\vec{\alpha},\vec{\beta};t)=\exp(t(\alpha_1e_1+\alpha_2e_2+\alpha_3e_3))$.
\end{remark}

We will need the following known proposition.

\begin{proposition}
\label{small}
If in the Lie group with left-invariant sub-Riemannian metric two
points are joined by two different normal geodesics of equal
length, then either of these geodesics is not a shortest arc or is not a part of a longer shortest
arc.
\end{proposition}

\begin{proposition}
\label{p1}	
Let $\beta>1$ and $\vec{\alpha}\cdot\vec{\beta}=0$. 
If the segment $\gamma(t)$, $0\leq t\leq T$, of a geodesic $\gamma(\vec{\alpha},\vec{\beta};t)$ is a shortest arc, then $T\leq \frac{2\pi}{\sqrt{\beta^2-1}}$. 
\end{proposition}

\begin{proof}
If $\beta>1$ and $\vec{\alpha}\cdot\vec{\beta}=0$ then in the notation of Corollary~1, we successively obtain
$w_1={\bf i}\sqrt{\beta^2-1}$, $w_2=\beta/2,$
$$m_1=\cos\frac{t\sqrt{\beta^2-1}}{2},\quad n_1=\frac{2\sin\frac{t\sqrt{\beta^2-1}}{2}}{\sqrt{\beta^2-1}},\quad m_2=\cos\frac{\beta t}{2},\quad n_2=\sin\frac{\beta t}{2}.$$
Therefore $m_1(t_0)=-1$, $n_1(t_0)=0$ for $t_0=\frac{2\pi}{\sqrt{\beta^2-1}}$. Then $c_i(t_0)=0$, $i=1,2,3,7;$
$$c_0(t_0)=-2\cos\frac{\beta\pi}{\sqrt{\beta^2-1}};\quad c_j(t_0)=\alpha_j\sin\frac{\beta\pi}{\sqrt{\beta^2-1}},\,\,j=4,5,6.$$
Consequently, $\gamma(\vec{\alpha},\vec{\beta};t_0)\in \SU(2)$ does not depend on the coordinates of the vector $\vec{\alpha}$.
It remains to apply Proposition \ref{small}.
\end{proof}

\begin{proposition}
\label{osn}
Let $a=\alpha_1e_1+\alpha_2e_2+\alpha_3e_3$, $b=\alpha_4e_4+\alpha_5e_5+\alpha_6e_6$. Set
$$x={\rm Re}\left(\frac{1}{2}\sqrt{(\alpha_1+{\bf i}\alpha_4)^2+(\alpha_2+{\bf i}\alpha_5)^2+(\alpha_3+{\bf i}\alpha_6)^2}\right),$$
$$y={\rm Im}\left(\frac{1}{2}\sqrt{(\alpha_1+{\bf i}\alpha_4)^2+(\alpha_2+{\bf i}\alpha_5)^2+(\alpha_3+{\bf i}\alpha_6)^2}\right).$$
For non-zero vectors $\vec{\alpha}=\{\alpha_1,\alpha_2,\alpha_3\}$, $\vec{\beta}=\{\alpha_4,\alpha_5,\alpha_6\}$, $\beta:=|\vec{\beta}|$,
the matrix $\exp(a+b)exp(-b)$ is Hermitian if and only if one of the conditions is satisfied:
	
1) the vectors $\vec{\alpha}$ and $\vec{\beta}$ are collinear;
	
2) $x=\cos(\beta/2)=\cos{y}=0$;
	
3) $\cos(\beta/2)\neq 0$, $\cos{y}\neq 0$ and the triples $x,y,\beta/2$ and ${\rm th}x,{\rm tg}y,{\rm tg}(\beta/2)$ are proportional;
	
4) $x=y=0$ and ${\rm tg}(\beta/2)=\beta/2.$
\end{proposition}

\begin{proof}
On the base of Proposition \ref{isom1}, we can assume that $a=\alpha e_1$, $\alpha>0$.	
It follows from (\ref{e1}), (\ref{e2}) that the matrix $\exp(a+b)exp(-b)$ is Hermitian if and only if
\begin{equation}
\label{oo}
\exp(b)\exp(a-b)=\exp(a+b)\exp(-b).
\end{equation}
Set
$$w_{+}=\frac{1}{2}\sqrt{\alpha^2-\beta^2+2{\bf i}\alpha\alpha_4},\quad w_{-}=\frac{1}{2}\sqrt{\alpha^2-\beta^2-2{\bf i}\alpha\alpha_4}.$$
Let $x={\rm Re}(w_{+})$, $y={\rm Im}(w_{+})$. Then, as is easy to see, $w_{-}=x-{\bf i}y$.
	
Let
$$z=\alpha^2-\beta^2+2\alpha\alpha_4{\bf i}:=r e^{\bf i\varphi}= r(\cos\varphi+{\bf i}\sin\varphi),\quad\mbox{where }
r^2=(\alpha^2-\beta^2)^2+4\alpha^2\alpha^2_4,$$
$$\cos\varphi=\frac{\alpha^2-\beta^2}{r},\,\,\, \sin\varphi=\frac{2\alpha\alpha_4}{r},\,\,\, \varphi=\sgn(\alpha_4)\arccos\frac{\alpha^2-\beta^2}{r}.$$
Then
\begin{equation}
\label{wpl} 
w_{+}= \frac{\pm\sqrt{r}}{2}(\cos(\varphi/2)+\sin(\varphi/2){\bf i}),\,\,\, 4xy=\alpha\alpha_4.
\end{equation}
	
Let us consider the case when $4xy=\alpha\alpha_4\neq 0$. Set
$$m_{+}={\rm ch}(w_{+})={\rm ch}x\cos{y}+{\bf i}{\rm sh}{x}\sin{y},\quad m_{-}={\rm ch}(w_{-})=\overline{m_{+}};$$
$$n_{+}=\frac{{\rm sh}(w_{+})}{w_{+}}=\frac{{\rm sh}(x+{\bf i}y)}{x+{\bf i}y}=\frac{x{\rm sh}x\cos{y}+y{\rm ch}x\sin{y}}{x^2+y^2}+{\bf i}\frac{x{\rm ch}x\sin{y}-y{\rm sh}x\cos{y}}{x^2+y^2},$$
$$n_{-}=\frac{{\rm sh}(w_{-})}{w_{-}}=\frac{{\rm sh}(x-{\bf i}y)}{x-{\bf i}y}=\overline{n_{+}}.$$
By virtue of the equalities (\ref{hh}) and the previous equality from the proof of Corollary 1, the left-hand side of (\ref{oo}) is equal to
$$(2\cos\frac{\beta}{2}e_0+\frac{2}{\beta}\sin\frac{\beta}{2}b)(2m_{-}e_0+n_{-}(\alpha e_1-b))=2\cos\frac{\beta}{2}m_{-}e_0+\frac{2}{\beta}\sin\frac{\beta}{2}m_{-}b$$
$$+\cos\frac{\beta}{2}n_{-}(\alpha e_1-b)+\frac{2}{\beta}\sin\frac{\beta}{2}n_{-}(\alpha(be_1)-b^2),$$
and the right side of (\ref{oo}) is equal to
$$(2m_{+}e_0+n_{+}(\alpha e_1+b))(2\cos\frac{\beta}{2}e_0-\frac{2}{\beta}\sin\frac{\beta}{2}b)=2\cos\frac{\beta}{2}m_{+}e_0-\frac{2}{\beta}\sin\frac{\beta}{2}m_{+}b$$
$$+\cos\frac{\beta}{2}n_{+}(\alpha e_1+b)-\frac{2}{\beta}\sin\frac{\beta}{2}n_{+}(\alpha(e_1b)+b^2).$$
Then it follows from (\ref{oo}) that
$$2\cos\frac{\beta}{2}{\rm Im}(m_{+}){\bf i}e_0-\frac{2}{\beta}\sin\frac{\beta}{2}{\rm Re}(m_{+})b+\alpha\cos\frac{\beta}{2}{\rm Im}(n_{+})e_1+
\cos\frac{\beta}{2}{\rm Re}(n_{+})b=$$
\begin{equation}
\label{kk}
\frac{1}{\beta}\sin\frac{\beta}{2}\left[n_{+}(\alpha (e_1b)+b^2)+n_{-}(\alpha (be_1)-b^2)\right].
\end{equation}
Let us calculate the expression in square brackets:
$n_{+}(\alpha (e_1b)+b^2)+n_{-}(\alpha (be_1)-b^2)=$
$$=({\rm Re}(n_{+})+{\bf i}{\rm Im}(n_{+}))(\alpha (e_1b)+b^2)+({\rm Re}(n_{+})-{\bf i}{\rm Im}(n_{+}))(\alpha (be_1)-b^2)$$
$$=\alpha{\rm Re}(n_{+})(e_1b+be_1)+{\bf i}{\rm Im}(n_{+})(\alpha(e_1b-be_1)+2b^2).$$
Note, using (\ref{ll}), that
$$e_1b+be_1=\alpha_4{\bf i}e_0,\quad \alpha(e_1b-be_1)+2b^2=\alpha(\alpha_5e_3-\alpha_6e_2)-\beta^2e_0.$$
By substitution the resulting expressions into (\ref{kk}) and equating the coefficients of the vectors ${\bf i}e_0,e_4,e_5,e_6$, and taking into account the equality $\alpha\alpha_4=4xy$, we obtain
\begin{equation}
\label{pp}
\cos\frac{\beta}{2}{\rm Im}(m_{+})=\frac{2xy}{\beta}\sin\frac{\beta}{2}{\rm Re}(n_{+})-\frac{\beta}{2}\sin\frac{\beta}{2}{\rm Im}(n_{+}),
\end{equation}
\begin{equation}
\label{pp2}
\alpha_4\left(\cos\frac{\beta}{2}{\rm Re}(n_{+})-\frac{2}{\beta}\sin\frac{\beta}{2}{\rm Re}(m_{+})\right)+\alpha\cos\frac{\beta}{2}{\rm Im}(n_{+})=0,
\end{equation}
$$\alpha_5\left(\cos\frac{\beta}{2}{\rm Re}(n_{+})-\frac{2}{\beta}\sin\frac{\beta}{2}{\rm Re}(m_{+})\right)=
-\frac{\alpha\alpha_6}{\beta}\sin\frac{\beta}{2}{\rm Im}(n_{+}),$$
$$\alpha_6\left(\cos\frac{\beta}{2}{\rm Re}(n_{+})-\frac{2}{\beta}\sin\frac{\beta}{2}{\rm Re}(m_{+})\right)=\frac{\alpha\alpha_5}{\beta}\sin\frac{\beta}{2}{\rm Im}(n_{+}).$$
From the last two equalities it follows that either $\alpha_5=\alpha_6=0$ or the coefficients of the variables $\alpha_5$, $\alpha_6$ are equal to zero.
In the first case, the vector $\vec{\beta}$ is collinear to $\vec{\alpha}$, and the matrix $\exp(a+b)\exp(-b)$ is Hermitian by Remark 4.
	
In the second case, taking into account (\ref{pp2}),
\begin{equation}
\label{jj}
{\rm Im}(n_{+})=0,\quad \cos\frac{\beta}{2}{\rm Re}(n_{+})=\frac{2}{\beta}\sin\frac{\beta}{2}{\rm Re}(m_{+}).
\end{equation}
It follows from the first equality in (\ref{jj}) that
$$\sin{y}=\frac{y{\rm sh}x\cos{y}}{x{\rm ch}x},\quad n_{+}=\frac{{\rm sh} x}{x}\cos{y}.$$
Then the equality (\ref{pp}) and the second equality in (\ref{jj}) can be written as
$$\frac{\beta}{2}\cos\frac{\beta}{2}\sin{y}=y\sin\frac{\beta}{2}\cos{y},\quad \frac{\beta}{2}\cos\frac{\beta}{2}{\rm sh}x=x\sin\frac{\beta}{2}{\rm ch}x.$$
Thus, for $\vec{\alpha}\not\parallel\vec{\beta}$ and $\vec{\alpha}\cdot\vec{\beta}\neq 0$ the matrix $\exp(a+b )\exp(-b)$ is Hermitian iff
$$\frac{{\rm th}x}{x}=\frac{{\rm tg}y}{y}=\frac{{\rm tg}(\beta/2)}{\beta/2}.$$
	
Let us consider the case $x=0$, $y\neq 0$. Then $\alpha_4=0,$ $m_{+}=\cos{y}$, $n_{+}=\frac{\sin{y}}{y}$ and the equality (\ref{kk}) is equivalent to the equality
$(\beta/2)\cos(\beta/2)\sin{y}=y\sin(\beta/2)\cos{y}.$ Then $\cos(\beta/2)=\cos{y}=0$  or
$$\frac{{\rm tg}y}{y}=\frac{{\rm tg}(\beta/2)}{\beta/2}.$$
	
Let us consider the case $x\neq 0$, $y=0$. Then $\alpha_4=0,$ $m_{+}={\rm ch}x$, $n_{+}=\frac{{\rm sh}x}{x}$ and the equality (\ref{kk}) is equivalent to the equality
$$\frac{{\rm th}x}{x}=\frac{{\rm tg}(\beta/2)}{\beta/2}.$$
	
Finally, let $x=y=0$. Then $m_{+}=n_{+}=1$ and the equality (\ref{kk}) is equivalent to the equality $${\rm tg}(\beta/2)=\beta/2.$$
	
Proposition \ref{osn} is proved.
\end{proof}

\section{Sub-Lorentzian nonspacelike abnormal extremals on $\GL^{+}(2,\mathbb{C})$}
\label{abn}

Let us search for abnormal extremals of the left-invariant sub-Lorentzian (anti)metric $d$ on $\GL^{+}(2,\mathbb{C})$.

By sequential substitution of $w=e_i$, $i=0,\dots,6$, into (\ref{hame2}), we get
\begin{equation}
\label{vw1}
\psi_0'=0,\quad\psi_1'=u_2\psi_6-u_3\psi_5,\quad \psi_2'=-u_1\psi_6+u_3\psi_4,\quad\psi_3'=u_1\psi_5-u_2\psi_4,
\end{equation}
\begin{equation}
\label{vw2}
\psi_4'=u_2\psi_3-u_3\psi_2,\quad \psi_5'=-u_1\psi_3+u_3\psi_1,\quad \psi_6'=u_1\psi_2-u_2\psi_1.
\end{equation}

Now let $M_0=0$ in (\ref{M1}). 
Then obviously $\psi_i(t)\equiv 0$, $i=0,1,2,3$, and taking into account (\ref{vw1}), (\ref{vw2}),
$\psi_k(t)\equiv -\alpha_k$, $k=4,5,6$, where $\alpha_4^2+\alpha_5^2+\alpha_6^2>0$, and for a measurable function $\varkappa(t)$, $t\in\mathbb{R}$,
\begin{equation}
\label{uk}
u_1(t)=\frac{\alpha_4{\rm sh}\varkappa(t)}{\sqrt{\alpha_4^2+\alpha_5^2+\alpha_6^2}},\quad 
u_2(t)=\frac{\alpha_5{\rm sh}\varkappa(t)}{\sqrt{\alpha_4^2+\alpha_5^2+\alpha_6^2}}, \quad
u_3(t)=\frac{\alpha_6{\rm sh}\varkappa(t)}{\sqrt{\alpha_4^2+\alpha_5^2+\alpha_6^2}},
\end{equation} 

\begin{equation}
\label{uz}	
u_0(t)={\rm ch}\varkappa(t),\,\,\mbox{если } u(t)\in\partial U.
\end{equation}

Some extremals can be both normal and abnormal with respect to different covector functions; such abnormal extremals are called nonstrictly abnormal. An abnormal extremal is called strictly abnormal if it is not nonstrictly abnormal.

\begin{proposition}
\label{nstrvr}
A parameterized by arclength timelike extremal in $(\GL^{+}(2,\mathbb{C}),d)$ is nonstrictly abnormal if and only if it is a 1-parameter subgroup with an initial unit vector in $(H, \langle\cdot,\cdot\rangle),$ or its left shift.
\end{proposition}

\begin{proof}
Due to the left-invariance of the sub-Lorentzian structure, we can assume that $g(0)=e$ for the considered extremal $g(t)$, $t\in\mathbb{R}.$	
	
Let us assume that this extremal is nonstrictly abnormal.
	
It follows from (\ref{uk}), (\ref{uz}) and the proof of Theorem \ref{main2} that
$$u(t)=g^{-1}(t)g'(t)=\alpha_0e_0-\sum\limits_{i=1}^{3}u_i(t)e_i,\quad \alpha_0=\sqrt{1+\alpha_1^2+\alpha_2^2+\alpha_3^2};\, \alpha_1,\alpha_2,\alpha_3\in\mathbb{R},$$
$$u_0(t)={\rm ch}\varkappa(t)\equiv\alpha_0,\quad \varkappa(t)\equiv\varkappa_0,\quad u_i(t)=\const=-\alpha_i,\, i=1,2,3.$$
It follows from (\ref{sg}) that
\begin{equation}
\label{onep} 
g(t)=\exp\left(t\sum_{i=0}^3\alpha_ie_i\right),\, t\in\mathbb{R},\quad \alpha_0=\sqrt{1+\alpha_1^2+\alpha_2^2+\alpha_3^2}.
\end{equation}
	
Conversely, let a 1-parameter subgroup in $\GL^{+}(2,\mathbb{C})$ of the form (\ref{onep}) be given.  
Then (\ref{onep}) is a parameterized by arclength timelike normal extremal  with constant control $u(t)=\sum_{i=0}^3\alpha_ie_i,$
defined according to Theorem \ref{main} by the constant covector function $\psi(t)=(\alpha_0,-\alpha_1,-\alpha_2,-\alpha_3,0,0,0)$.
	
It is easy to check, using (\ref{vw1}), (\ref{vw2}), that the extremal (\ref{onep}) is abnormal with respect to the covector function
$$\psi(t)=\left(0,0,0,0,\frac{-\alpha_1}{\sqrt{\alpha_0^2-1}},\frac{-\alpha_2}{\sqrt{\alpha_0^2-1}},\frac{-\alpha_3}{\sqrt{\alpha_0^2-1}}\right),\quad\mbox{if }\,\alpha_0>1;$$
$$\psi(t)=(0,0,0,0,-\alpha_4,-\alpha_5,-\alpha_6),\,\,\alpha_4^2+\alpha_5^2+\alpha_6^2>0,\quad\mbox{if }\,\alpha_0=1.$$
In the last case we have $\varkappa(t)\equiv 0.$
\end{proof}

Note that due to Proposition \ref{nstrvr}, for any nonconstant smooth function $\varkappa(t),$ $t\in\mathbb{R},$ defining the control 
$u(t)=u_0(t)e_0-\sum\limits_{i=1}^{3}u_i(t)e_i$ by formulas (\ref{vw1}), (\ref{vw2}), a solution $g=g(t)$ of the ordinary differential equation (\ref{sg}) with the initial condition $g(0)=e$ (with the maximum possible connected domain $t\in J\subset\mathbb{R}$) is a strictly abnormal extremal. For example, in the case 
$\alpha_4=\alpha_5=0,$ $\alpha_6=1,$ $\varkappa(t)=t$ we get the strictly abnormal extremal
$$g(t)=\exp\left(2{\rm sh}(t/2)e_0+2(1-{\rm ch}(t/2))e_3\right)=\left(\begin{array}{cc}
\exp(1-e^{-t/2}) & 0  \\
0 & \exp(e^{t/2}-1)
\end{array}\right).$$

It is natural to assume that an abnormal isotropic extremal is defined by equations (\ref{vw1}), (\ref{vw2}) and
\begin{equation}
\label{ukisotr}
u_1(t)=\frac{\alpha_4\varkappa(t)}{\sqrt{\alpha_4^2+\alpha_5^2+\alpha_6^2}},\quad
u_2(t)=\frac{\alpha_5\varkappa(t)}{\sqrt{\alpha_4^2+\alpha_5^2+\alpha_6^2}}, \quad
u_3(t)=\frac{\alpha_6\varkappa(t)}{\sqrt{\alpha_4^2+\alpha_5^2+\alpha_6^2}},\quad
\end{equation}
\begin{equation}
\label{uzisotr}	
u_0(t)=|\varkappa(t)|>0
\end{equation}
provided that $\alpha_1=\alpha_2=\alpha_3=0$ and $\alpha_4^2+\alpha_5^2+\alpha_6^2>0.$ 

Then, similarly to Proposition \ref{nstrvr}, the following proposition is proved.

\begin{proposition}
\label{nstrisotr}	
An isotropic extremal in $(\GL^{+}(2,\mathbb{C}),d)$ is nonstrictly abnormal if and only if it is a 1-parameter subgroup with an initial nonzero isotropic vector in $(H,\langle \cdot,\cdot\rangle)$ or its left shift.
\end{proposition}

Assuming $\alpha_4=\alpha_5=0$, $\alpha_6=1,$ $\varkappa(t)=\frac{e^t}{2},$ we obtain a strictly abnormal isotropic extremal
$$g(t)=\exp((e^{t/2}-1)(e_0-e_3))=
\left(\begin{array}{cc}
1 & 0  \\
0 & \exp(e^{t/2}-1)
\end{array}\right),\, t\in\mathbb{R}.$$

On base of Theorem \ref{onepa} and Propositions \ref{nstrvr}, \ref{nstrisotr} we obtain

\begin{corollary}
\label{nonstr}
Each segment of a nonspacelike nonstrictly abnormal extremal is a longest arc in $(\GL^+(2,\mathbb{C}),d)$, parametrized proportionally to the arclength for the timelike extremal.
\end{corollary}

\section{Addition about Lie groups related to the group $\GL^+(2,\mathbb{C})$}
\label{append}

The following inclusions of Lie groups are valid:

$${\rm SU}(2)\subset \SL(2,\mathbb{C})\subset \GL^+(2,\mathbb{C})\subset \GL(2,\mathbb{C}),$$

$${\rm U}(2)\subset S^1I\cdot \SL(2,\mathbb{C})\subset \GL(2,\mathbb{C}),$$
where $S^1=(\{z\in \mathbb{C}: |z|=1\},\cdot)$ is a multiplicative group, $S^1I$ is a group of
diagonal $(2\times 2)$-matrices with elements $z\in S^1$ on the diagonal; 
$S^1I\cdot\SL(2,\mathbb{C})$ consists of matrices in $\GL(2,\mathbb{C})$ with modulus of determinant equal to $1.$
Moreover, ${\rm SU}(2)$ is the maximal compact Lie subgroup in $\SL(2,\mathbb{C}),$ $\GL^+(2,\mathbb{C});$ ${\rm U}(2)=S^1I\cdot {\rm SU}(2)$ is a maximal compact Lie subgroup in $S^1I\cdot\SL(2,\mathbb{C}),$ $\GL(2,\mathbb{C}).$

The Lie group $\SL(2,\mathbb{C})$ is an algebraic linear group (see Example 1, Section 3.1 in \cite{VinOn88});
$\SL(2,\mathbb{C})$ is a self-adjoint group: $g^{\ast}\in \SL(2,\mathbb{C})$ if $g\in\SL(2,\mathbb{C}).$
Due to Theorem 6.6, item ''6.4. Polar decomposition'' in \cite{VinGorbOn90} the following theorem is true. 

\begin{theorem}
\label{dec}
There exists a polar decomposition $\SL(2,\mathbb{C})=P\cdot{\rm SU}(2),$ where $P$ is a submanifold of $\SL(2,\mathbb{C}) $ consisting of all positive
definite self-adjoint operators, and the representation
$$g=p\cdot k,\,\,\mbox{where } g\in\SL(2,\mathbb{C}),\,\,p\in P,\,\,k\in{\rm SU}(2),$$
is the only one. More precisely,
the mapping $\varphi: H_0\times{\rm SU}(2)\rightarrow\SL(2,\mathbb{C}),$ given by the formula $\varphi(y,k)=\exp(y)k,$ is a diffeomorphism.
\end{theorem}

Proposition \ref{iso}, Theorem \ref{dec} and Theorem 1 in \cite{Ber23} imply

\begin{corollary}
\label{quat}
There exists a diffeomorphism $\GL^+(2,\mathbb{C})=P\cdot(\mathbb{R}_+I\times{\rm SU}(2))$ (polar decomposition), where the Lie group 
$\mathbb{R}_+I\times{\rm SU}(2)$ is isomorphic to the universal covering $\tilde{{\rm U}}(2)$ of the Lie group ${\rm U}(2)$ and the multiplicative group $\mathbb{H}_0$ of nonzero quaternions.
\end{corollary}
	
\begin{corollary}
\label{pos}
The mapping $\exp: H_0\rightarrow\exp(H_0)=P$ is a diffeomorphism.
\end{corollary}

Since ${\rm SU}(2)$ is diffeomorphic to $S^3,$ it follows from the above that the fundamental groups $\pi_1$ of the Lie groups $\GL^{+}(2,\mathbb{C})$ and $\SL(2,\mathbb{C})$ (respectively $\GL(2,\mathbb{C})$ and $S^1I\cdot\SL(2,\mathbb{C})$) are isomorphic to the fundamental groups $\pi_1$ of the
Lie group ${\rm SU}(2)$ (respectively ${\rm U}(2)$).
Therefore all Lie groups ${\rm SU}(2),$ $\SL(2,\mathbb{C}),$
$\GL^+(2,\mathbb{C})$ are simply connected, and
$$(\mathbb{Z},+)=\pi_1(U(2))=\pi_1(S^1\cdot\SL(2,\mathbb{C}))=\pi_1(\GL(2, \mathbb{C})).$$

The following theorem follows from Corollary \ref{pos} and Proposition \ref{iso}.

\begin{theorem}
\label{h+}
Let $H^{+}$ be a manifold of all positive definite Hermitian matrices in $\GL^+(2,\mathbb{C})$.
The exponential map $\exp:H\rightarrow H^{+}$ is a diffeomorphism.
\end{theorem}

The vector space ${\bf i}H$ over $\mathbb{R}$ of all skew-Hermitian complex $(2\times 2)$-matrices with a quadratic form $\det f,$ $f\in {\bf i}H ,$ which we shall denote in this section as $\langle f,f\rangle,$ is a Minkowski space-time with the signature $(-,+,+,+)$, and $({\bf i}H=\mathfrak{u}(2),[\cdot,\cdot])$ is the Lie algebra of the Lie groups ${\rm U}(2)$ and $\tilde{U}(2).$

The corresponding left-invariant Lorentzian metric $\langle\cdot,\cdot\rangle$ on ${\rm U}(2)$ and on
$\tilde{{\rm U}}(2)\cong (\mathbb{R}_+,\cdot)\times {\rm SU}(2)$ is biinvariant. Moreover, $((\mathbb{R}_+,\cdot)\times {\rm SU}(2),\langle\cdot,\cdot\rangle)$ is isometric to Einstein Universe (space-time) with induced biinvariant Riemannian metric of the constant unit sectional curvature on ${\rm SU}(2),$ diffeomorphic to $S^3$ (Theorem 4 in \cite{Ber23}).

In Theorems 5, 7, 10 in \cite{Ber23}, the author consructed a {\it stratification} of smooth manifold ${\rm U}(2),$ i.e. strictly decreasing sequence of closed subsets ({\it strata}) in ${\rm U}(2)$ $X_0={\rm U}(2)\supset X_1\supset X_2$  such that $X_k\setminus X_{k+1},$ $k=0,1,$ are 
smooth submanifolds. For this stratification $N:={\rm U}(2)\setminus X_1$ is everywhere dense in ${\rm U}(2)$ and diffeomorphic to ${\bf i}H,$ $X_1={\rm U}(2)\setminus N$ is homeomorphic to $S^3$ with a pair of diametrically opposite points identified at one point $x_0$,
$X_2=\{x_0\},$ $X_1\setminus X_2$ is diffeomorphic to $S^2\times (0,1).$

Moreover, $X_1$ is the union of all closed, diffeomorphic to circle,
isotropic geodesics in $({\rm U}(2),\langle\cdot,\cdot\rangle)$ with the origin $-I;$ physicists call the stratum $X_1$  the {\it conformal infinity of the
Minkowski space} $M_0.$

The real linear span of the union $H\cup {\bf i}H$ is the $8$-dimensional associative algebra $\M(2,\mathbb{C})$ of all complex $(2\times 2)$-matrices.

It is clear that $(\M(2,\mathbb{C}),[\cdot,\cdot])=(\mathfrak{gl}(2,\mathbb{C}),[\cdot,\cdot]).$ 

Algebra $\M(2,\mathbb{C})$ is the {\it Clifford algebra} $Cl_3$ (see~\cite{Loun01}, pp.189, 190), since the orthonormal with respect to scalar product $(\cdot,\cdot)=(-1/4)\langle\cdot,\cdot\rangle$ basis $\sigma_1,\sigma_2,\sigma_3$ of Euclidean space $\mathbb{R}^3$ is a (minimal) generating set for $\M(2,\mathbb{C})$ and $\sigma_l\cdot\sigma_k+\sigma_k\cdot\sigma_l=2\delta_{lk}I.$

It follows that
\begin{equation}
\label{cl} 
u,v\in (\mathbb{R}^3,(\cdot,\cdot))\,\, \Longrightarrow\,\, \{u,v\}:= uv+vu=2 (u,v)I.
\end{equation}
Analogue of (\ref{cl}) is the only condition for the coordinateless definition of the Clifford algebra $Cl_n,$ being an associative algebra
of dimension $2^n,$ $n\geq 1,$ over the field $\mathbb{R},$ generated by Euclidean space $(\mathbb{R}^n,(\cdot,\cdot));$
the matrix $I$ is replaced by the multiplicative unit ${\bf 1}$ of the algebra $Cl_n.$

\begin{example}
The Clifford algebra $Cl_1$ is isomorphic to the subalgebra in $Cl_3,$ generated by any basis vector $\{\sigma_l, l=1,2,3\}$ in $\mathbb{R}^3,$
the Clifford algebra $Cl_2$ is isomorphic to the subalgebra in $Cl_3,$ generated by any pair of basis vectors $\{\sigma_l, l=1,2,3\}$.
In particular, the Clifford algebra $Cl_2$ generated by the matrices $\sigma_1, \sigma_3,$ is the algebra $\M(2,\mathbb{R})$ of real $(2\times 2)$-matrices;
the corresponding Lie algebra $(\M(2,\mathbb{R}),[\cdot,\cdot])=(\mathfrak{gl}(2,\mathbb{R}),[\cdot,\cdot]) $ is the Lie algebra of the Lie group
$\GL(2,\mathbb{R})\subset \GL(2,\mathbb{C})$. For corresponding unimodular Lie subgroups, we get $\SL(2,\mathbb{R})\subset \SL(2,\mathbb{C}).$
\end{example}	

\begin{remark}
\label{negcl}
Under replacement the scalar product $(\cdot,\cdot)$ in (\ref{cl}) with the pseudoscalar product $\langle\cdot,\cdot\rangle=-(\cdot,\cdot),$
$I$ on ${\bf 1}$ and 3 on $n\geq 1,$ we obtain the definition of the Clifford algebra $Cl_{0,n},$ generated by the space $\mathbb{R}^{0,n}$ (see \cite{Loun01}, p.~189), and isomorphisms $Cl_{0,1}\cong \mathbb{C},$ $Cl_{0,2}\cong \mathbb{H},$ $Cl_{0,3}\cong \mathbb{H}\oplus \mathbb{H}.$
\end{remark}
   
The unimodular Lie group $\SL(2,\mathbb{R})$ doubly covers the restricted Lorentz group $\SO_0(1,2)$ in the same way as $\SL(2,\mathbb{C})$ doubly covers
the Lie group $\SO_0(1,3).$ But \linebreak $\pi_1(\SL(2,\mathbb{C}))=0,$ while $\pi_1(\SL(2,\mathbb{R}))\cong(\mathbb{Z},+).$

Moreover $\SL(2,\mathbb{R})/\SO(2)=L^2$ is a Riemannian symmetric space, Lobachevsky plane. 
Therefore, the results of paper \cite{Ber18} are applicable to the corresponding left-invariant sub-Riemannian metric on $\SL(2,\mathbb{R})$.
Earlier in papers \cite{Ber15}, \cite{BerZub16}---\cite{BerZub17}, were studied these metrics, in particular, their geodesics, shortest paths, distances and etc. on the Lie group $\SL(2,\mathbb{R})$ and the Lie groups related to it by coverings.

In papers \cite{AgrBarRiz18} and \cite{Ber17}, \cite{Ber18} were suggested alternative definitions of Ricci and sectional curvatures of homogeneous sub-Riemannian manifolds; the second one is based on Solov'ev approach to definition of curvatures for rigged distributions \cite{Sol84}.

Notice that every nontrivial element of any 1-parameter subgroup from Theorem \ref{ppp} is a {\it boost}
\cite{RumFet77}, \cite{Lap21}.
Part I. ''Spinor algebra'' in \cite{RumFet77} contains the section ''Revolutions and boosts'', in which, for a given (not unique) choice of the time axis in the Minkowski space-time $M_0$, two types of Lorentz transformations are considered. Elements of the group $\SO_0(1,3),$ preserving events of the time axis are called {\it revolutions}, and nontrivial elements of this group that do not change vectors from $M_0,$ orthogonal to some 2-plane in $M_0,$ containing the time axis, are called {\it boosts}. It is clear that the revolutions form a group isomorphic to $\SO(3).$
Let us quote verbatim from \cite{RumFet77}: ''Notice that boosts do not form a group: although for each boost $B$ the inverse transformation is again a boost, the product of two boosts, generally speaking, is not a boost.'' Let us add that boosts corresponding to any one chosen time axis generate the group $\SO_0(1,3)$ since the Lie group $\SL(2,\mathbb{C})$ is generated by elements in the set $\exp(H_0)$ of positive definite Hermitian matrices from $\SL(2,\mathbb{C})$.

The book \cite{BEE96} (its first author is a pupil of H.~Busemann) introduces the concepts of {\it chronological} and {\it causal} 
structures of the time-oriented Lorentzian space-times $(M,g)$ and their classification with respect to these structures.

Let $p,q\in (M,g).$ We will write $p<<q$ (resp., $p\leq q$), if there exists a future-directed timelike (respectively, non-spacelike)
piecewise smooth curve in $(M,g)$ from $p$ to $q$ (it is assumed that $p\leq p$).

{\it The chronological past and future of the point} $p$ are the sets\\ $I^-(p)=\{q\in M\mid q<< p\}$ and $I^+(p)=\{q\in M\mid p<< q\}$ respectively. 

{\it The causal past and future of the point} $p$ are the sets\\ $J^-(p)=\{q\in M\mid q\leq p\}$ and $J^+(p)=\{q\in M\mid p\leq q\}$ respectively.

The sets $I^-(p),$ $I^+(p)$ are always open in any Lorentzian space-time, but in general, $J^-(p),$ $J^+(p)$ neither open nor closed \cite{BEE96}. 

The space--time $(M,g)$ is called {\it strictly causal} if the {\it Alexandrov topology}, whose base consists of the {\it intervals} $I^+(p)\cap I^-(q),$
$p,q\in M$, is Hausdorff. The space--time $(M,g)$ is called {\it globally hyperbolic}, if $(M,g)$ is strictly causal and all sets $J^+(p)\cap J^-(q),$
$p,q\in M$, are compact.

Due to papers \cite{Ger70} and \cite{BerSan03}, the space--time $(M,g)$ is  globally hyperbolic iff  
there exists a diffeomorphism $f:M\leftrightarrow \mathbb{R}\times S$ such that for every $t\in\mathbb{R},$ $f^{-1}(\{t\}\times S)$ is a smooth regular spacelike hypersurface in $(M,g)$ (called the {\it Cauchy surface}).

Thus, from this and Proposition  \ref{iso} it follows that the Lorentzian space--time 
\linebreak $(\GL^+(2,\mathbb{C}),\langle\cdot,\cdot\rangle)$  is globally hyperbolic.

The chronological past and future, the causal past and future, and hence the global hyperbolicity, are defined in exactly the same way for an arbitrary (time-oriented) sub-Lorentzian space-time.

It is clear that for each element $p\in \GL^+(2,\mathbb{C}),$ the causal past $J_d^-(p)$ (the causal future $J_d^+(p)$) of the space $(\GL^+(2,\mathbb{C}),d)$ is strictly included in the set $J^-(p)$ (resp. $J^+(p)$) for $(\GL^+(2,\mathbb{C}),\langle\cdot,\cdot\rangle).$ 
Moreover, it follows from Theorem \ref{slgeod} that the sets $J_d^-(p)$, $J_d^+(p)$ are closed. All this implies that $(\GL^+(2,\mathbb{C}),d)$ is globally hyperbolic.

The following general question naturally arises.
\begin{question}
\label{question}
Is the above statement from \cite{Ger70} and \cite{BerSan03} true for sub-Lorentzian space-time, particularly, for Lie groups with left-invariant sub-Lorentzian (anti)metrics?
\end{question}

\end{document}